\documentclass[11pt, letterpaper]{amsart}
\usepackage{amsmath,amssymb}
\usepackage{hyperref}

\usepackage{tikz}
\usetikzlibrary{arrows.meta,calc,decorations.markings,cd}

\usepackage[bb=ams,scr=euler]{mathalpha}
\usepackage{enumitem}
\setenumerate{label=(\arabic*)}

\usepackage{thmtools}
\declaretheoremstyle[headfont=\normalsize\normalfont\bfseries,notefont=\mdseries,
notebraces={(}{)},bodyfont=\normalfont\itshape,postheadspace=0.5em]{italstyle}

\declaretheorem[style=italstyle,name=Theorem]{theorem}
\declaretheorem[style=italstyle,name=Lemma,sibling=theorem]{lemma}
\declaretheorem[style=italstyle,name=Definition,sibling=theorem]{definition}
\declaretheorem[style=italstyle,name=Proposition,sibling=theorem]{proposition}

\declaretheoremstyle[headfont=\normalsize\normalfont\bfseries,notefont=\mdseries,
notebraces={(}{)},bodyfont=\normalfont,postheadspace=0.5em]{normalstyle}
\declaretheorem[style=normalstyle,name=Remark,sibling=theorem]{remark}

\makeatletter
\newcommand{\abs}[1]{\left|#1\right|}
\newcommand{\bd}{\partial}
\newcommand{\C}{\mathbb{C}}
\renewcommand{\d}{\mathrm{d}}
\newcommand{\id}{\mathrm{id}}

\newcommand{\ip}[1]{\left\langle#1\right\rangle}
\newcommand{\pd}[2]{\frac{\partial #1}{\partial #2}}
\newcommand{\R}{\mathbb{R}}
\newcommand{\set}[1]{\left\{#1\right\}}
\newcommand{\Z}{\mathbb{Z}}

\renewcommand\section{\@startsection{section}{1}{0pt}{-3.5ex \@plus -1ex \@minus -.2ex}{2.3ex \@plus.2ex}{\centering\itshape}}
\renewcommand{\subsection}{\@startsection{subsection}{2}\z@{.5\linespacing\@plus.7\linespacing}{-.5em}{\normalfont\itshape}}
\makeatother

\author{Dylan Cant}
\email{dylan@dylancant.ca}
\address{Départment de mathématiques et de statistique, Université de Montreal, Pavillon André-Aisenstadt, 2920 Chemin de la Tour, Montreal, Quebec, H3T 1J4, Canada}

\author{Julio Sampietro Christ}
\email{julio.sampietro-christ@universite-paris-saclay.fr}
\address{Institut de mathématique d'Orsay, Université Paris-Saclay, Bâtiment 307, rue Michel Magat, F-91405 Orsay Cedex, France}

\date{\today}

\begin{document}
\title{Equivariant Lagrangian displacements}
\begin{abstract}
  This paper proves that certain monotone Lagrangians in the standard symplectic vector space cannot be displaced by a Hamiltonian isotopy which commutes with the antipodal map. The method of proof is to develop a Borel equivariant version of the quantum cohomology of Biran and Cornea, and prove it is sensitive to equivariant displacements. The Floer--Euler class of Biran and Khanevsky appears as a term in the equivariant differential in certain cases.
\end{abstract}
\maketitle

\section{Introduction}
\label{sec:introduction}

Suppose that $L$ is a Lagrangian submanifold of a symplectic manifold $W$ admitting a symplectic involution $a$, and suppose $L$ is preserved by the involution as a set. The question we are interested in is the following:

\emph{Does there exist an equivariant Hamiltonian isotopy $\varphi_{t}$ which displaces $L$?}

Here \emph{equivariant} is in the sense $a\varphi_{t}=\varphi_{t}a$ holds for all $t$. If the answer is ``yes'' then we say $L$ is \emph{equivariantly displaceable}.

The prototypical example is $W=\C^{n}$, $a(z)=-z$; here one can consider Lagrangians $L$ of the following form: if $\pi:S^{2n-1}\to \mathbb{C}P^{n-1}$ is the standard Hopf fibration, then $L=\pi^{-1}(\bar{L})$ is invariant under $a$ for any Lagrangian submanifold $\bar{L}\subset \mathbb{C}P^{n-1}$. As an application of our methods, we will prove:
\begin{theorem}\label{theorem:application}
  The Lagrangian $L=\pi^{-1}(\bar{L})$ in $\C^{n}$ is not equivariantly displaceable, with respect to $a(z)=-z$, provided both:
  \begin{enumerate}
  \item $\bar{L}$ is monotone with minimal Maslov number at least $2$,
  \item the quantum cohomology of $\bar{L}$ with coefficient field $\Z/2\Z$, in the sense of \cite{biran-cornea-arXiv-2007,biran-cornea-GT-2009}, is non-vanishing.
  \end{enumerate}
\end{theorem}
\begin{remark}
  In the case when $n=1$, the Lagrangian $L$ is simply the unit circle, and an elementary argument can be used to prove Theorem \ref{theorem:application}. We give such an argument in \S\ref{sec:elem-argum-prov}, although we are confident readers can furnish a proof without much difficulty.
\end{remark}
\begin{remark}
  It is known that the quantum cohomology of $\bar{L}$ is non-vanishing in various cases, e.g., if $\bar{L}$ is the real projective space or the Clifford torus (in this case $L$ is just the product of $n$ circles); see \cite{biran-cornea-arXiv-2007,biran-cornea-GT-2009} and \cite{cho-imrn-2004}.
\end{remark}

\subsection{Quantum cohomology for monotone Lagrangians}
\label{sec:quant-cohom}

Throughout we work in the context of a single monotone Lagrangian $L\subset W$, with minimal Maslov number $N\ge 2$; here $W$ should be tame at infinity as in \cite{gromov-inventiones-1985}.

Recall that the \emph{quantum cohomology} for a monotone Lagrangian $L$ (with minimal Maslov number at least $2$) is a deformation of the Morse cohomology of $L$. See \cite{biran-cornea-GT-2009} for an overview of the construction. To define the quantum cohomology, one requires:
\begin{itemize}
\item a Morse-Smale pseudogradient $P$ on $L$,
\item an $\omega$-tame almost complex structure $J$ on $W$.
\end{itemize}
One defines a differential on the vector space:
\begin{equation*}
  \mathrm{CM}(P)\otimes \mathbf{k}[q^{-1},q]]
\end{equation*}
where $\mathrm{CM}(P)$ is the free $\mathbf{k}$-vector space generated by the zeros of $P$ (graded by the Morse index), and $q$ is a formal variable of degree $N$. Throughout this paper, we work with the coefficient field $\mathbf{k}=\Z/2\Z$.

The differential is a polynomial expression:
\begin{equation*}
  d=d_{0}+qd_{1}+q^{2}d_{2}+\cdots,
\end{equation*}
where $d_{i}:\mathrm{CM}(P)\to \mathrm{CM}(P)$ is a linear endomorphism of degree $1-Ni$. Moreover, $d_{0}$ is the classical Morse differential associated to the pseudogradient $P$. The quantum terms $d_{i}$, $i>0$, are defined by counting \emph{pearl trajectories}, and are described in detail in \S\ref{sec:recoll-lagr-quant}. As we will explain later, these terms depend on a choice of almost complex structure $J$.

As established by Biran and Cornea, provided $J$ is generic, the total differential $d=d_{0}+qd_{1}+\cdots$ squares to zero; the resulting homology vector space is denoted by $\mathrm{QH}^{*}(L;P,J)$. The significance of this invariant is:
\begin{proposition}
  If $\mathrm{QH}^{*}(L;P,J)\ne 0$, then $L$ cannot be displaced by a Hamiltonian isotopy.\hfill$\square$
\end{proposition}

The quantum cohomology is known to be isomorphic to the Lagrangian Floer cohomology of $L$ with itself; see \cite[Theorem A(vi)]{biran-cornea-GT-2009}. As in \cite{biran-cornea-GT-2009}, in this paper we do not make any use of the fact that the Lagrangian quantum homology can be identified with the Floer homology.

\subsection{Equivariant quantum cohomology}
\label{sec:equiv-quant-cohom}

Suppose now that $a:W\to W$ is a symplectic involution and $L$ is preserved by $a$. Our approach to $\Z/2\Z$-equivariant quantum cohomology is a ``pearl trajectory'' version of the equivariant cohomology of \cite{seidel-smith-GAFA-2010,seidel-GAFA-2015}, which is itself a Morse/Floer version of the Borel equivariant cohomology. As in the ordinary quantum cohomology, the definition depends on some auxiliary choices. As part of the data, one picks $P,J$ exactly as in \S\ref{sec:quant-cohom}. The underlying equivariant complex is:
\begin{equation*}
  \mathrm{CM}(P)\otimes \mathbf{k}[q^{-1},q]]\otimes \mathbf{k}[e],
\end{equation*}
and the differential $d_{eq}$ is a polynomial expression:
\begin{equation*}
  d_{eq}:=\sum q^{i}e^{j}d_{i,j},
\end{equation*}
where $d_{i,j}:\mathrm{CM}(P)\to \mathrm{CM}(P)$ is a graded endomorphism of degree $1-Ni-j$.

The definition will be such that the specialization $e=0$ is the ordinary quantum cohomology differential; the terms appearing with coefficient $e^{j}$ form a graded endomorphism of $\mathrm{CM}(P)\otimes \mathbf{k}[q^{-1},q]]$ of degree $1-j$.

The terms $d_{i,j}$ are constructed by following the recipe of \cite{seidel-smith-GAFA-2010,seidel-GAFA-2015}. For now, we will just say that $d_{i,j}$ depends on the choice of sufficiently regular data $(P_{\eta},J_{\eta})$ parametrized by $\eta\in S^{\infty}$. These data are required to restrict to $(P,J)$ at various points on the sphere. All of this is explained properly in \S\ref{sec:equiv-diff}.
For such regular data $(P_{\eta},J_{\eta})$, we define $\mathrm{QH}^{*}_{eq}(L;P_{\eta},J_{\eta})$ to be the homology of this complex.

\begin{remark}
  It is important to note that $\mathrm{QH}^{*}_{eq}$ is an \emph{uncompleted} equivariant cohomology, in the sense that it is a module over $\mathbf{k}[e]$ rather than over the ring $\mathbf{k}[[e]]=H^{*}(\mathrm{RP}^{\infty})$. This algebraic curiosity plays an important role in our proof of Theorem \ref{theorem:application}.
\end{remark}

\begin{remark}[Invariance statement]\label{remark:invariance-statement}
  Well-established continuation arguments then prove the map sending regular Borel data $(P_{\eta},J_{\eta})$ to $\mathrm{QH}_{eq}^{*}(L;P_{\eta},J_{\eta})$ can be equipped with the structure of a \emph{connected simple system}\footnote{A \emph{connected simple system} is a functor from an indiscrete groupoid.} \cite{salamon-trans-ams-1985}. In other words, if one considers the set of regular data as an indiscrete groupoid (a category with a unique morphism between any two elements) then:
  \begin{equation*}
    (P_{\eta},J_{\eta})\mapsto \mathrm{QH}_{eq}^{*}(L;P_{\eta},J_{\eta})
  \end{equation*}
  extends to a functor; we explain the definition of the morphisms and prove this invariance in \S\ref{sec:cont-isom}. The limit\footnote{In the category theory sense, see \cite{maclean-cat-working-math}} of this connected simple system is denoted $\mathrm{QH}^{*}_{eq}(L)$, and is the key player in our paper.
\end{remark}

\subsection{Equivariant displacements}
\label{sec:equiv-displ}

Equivariant quantum cohomology is sensitive to equivariant displacements:
\begin{theorem}\label{theorem:displacement}
  If $\psi_{t}$ is an equivariant Hamiltonian isotopy displacing $L$ then the equivariant quantum cohomology vanishes $\mathrm{QH}^{*}_{eq}(L)=0$.
\end{theorem}
The proof of the theorem is given in \S\ref{sec:proof-theorem-displacement} and goes as follows: first one shows:
\begin{equation*}
  \id:\mathrm{QH}^{*}_{eq}(L;P_{\eta},J_{\eta})\to\mathrm{QH}^{*}_{eq}(L;P_{\eta},J_{\eta})
\end{equation*}
is represented as a certain chain level continuation map. The continuation data can be deformed using the equivariant isotopy $\psi_{t}$, and, under the hypothesis that $\psi_{1}(L)\cap L=\emptyset$, one can prove that the deformed continuation map is zero. Thus the identity map is zero, as desired.

\subsection{The Biran-Khanevsky Floer-Gysin sequence}
\label{sec:biran-khanevnsky-floer-gysin-sequence}

The paper \cite{biran-khanevsky-CMH-2013} considers the quantum cohomology of Lagrangians $L\in \bd \Omega$ where $\Omega$ is a Liouville domain and $\bd \Omega$ has a Reeb flow inducing a free $\R/\Z$-action. Throughout this section we assume $\pi_{2}(\Omega,\bd\Omega)=0$, in addition to the usual assumptions that $L$ is monotone with minimal Maslov number at least $2$.

In this setting the \emph{reduced Lagrangian} $\bar{L}=L/(\R/\Z)$ is a monotone Lagrangian in the symplectic reduction $M=\bd\Omega/(\R/\Z)$; moreover $\bar{L}$ and $L$ have the same minimal Maslov number. The paper \cite{biran-khanevsky-CMH-2013} constructs an element $F\in \mathrm{QH}^{2}(\bar{L})$ called the \emph{Floer-Euler class}, and proves an exact triangle (the \emph{Floer-Gysin sequence}):
\begin{equation*}
  \begin{tikzcd}
    \cdots\arrow[r]&{\mathrm{QH}(\bar{L})}\arrow[r,"{F}"] &{\mathrm{QH}(\bar{L})}\arrow[r,"{}"] &{\mathrm{QH}(L)}\arrow[r]&\cdots
  \end{tikzcd}
\end{equation*}
where $F$ as a map denotes the product with $F$ with respect to Biran-Cornea's quantum multiplication \cite{biran-cornea-arXiv-2007,biran-cornea-GT-2009}. An example of such domain $\Omega$ is:
\begin{equation*}
  \Omega=B(1)=\set{z\in \C^{n}:\pi\abs{z}^{2}\le 1},
\end{equation*}
and this has reduction $M=\mathbb{C}P^{n-1}$; this is relevant to Theorem \ref{theorem:application}.

\subsection{Main technical result}
\label{sec:main-techn-result}

We prove the following:
\begin{theorem}\label{theorem:main-computation}
  In the context of \S\ref{sec:biran-khanevnsky-floer-gysin-sequence}, if the $\Z/2\Z$-action on $\bd\Omega$ inherited from its $\R/\Z$-action extends to an exact symplectic involution of $\Omega$, then:
  \begin{equation}\label{eq:main-isomorphism}
    \mathrm{QH}_{eq}^{*}(L)\simeq (\mathrm{QH}^{*}(\bar{L})\otimes \mathbf{k}[e])/(e^{2}+F)\simeq \mathrm{QH}^{*}(\bar{L})\oplus \mathrm{QH}^{*-1}(\bar{L}),
  \end{equation}
  Here $(e^{2}+F)$ is the image of the morphism $e^{2}+F$, i.e., the ideal generated by $e^{2}+F$ in the commutative algebra $\mathrm{QH}^{*}(\bar{L})\otimes \mathbf{k}[e]$.
\end{theorem}

\begin{remark}
  If $\mathrm{QH}^{*}(L)=0$, e.g., if $\mathrm{SH}(\Omega)=0$, then $F$ must be a unit by the Floer-Gysin sequence. Then, if we extend coefficients $\mathbf{k}[e]\to \mathbf{k}[[e]]$ it holds:
  \begin{equation*}
    \mathrm{QH}_{eq}^{*}(L)\otimes_{\mathbf{k}[e]} \mathbf{k}[[e]]\simeq (\mathrm{QH}^{*}(\bar{L})\otimes \mathbf{k}[[e]])/(e^{2}+F)=0,
  \end{equation*}
  since $e^{2}+F$ is a unit in $\mathrm{QH}^{*}(\bar{L})\otimes \mathbf{k}[[e]]$. However, the same conclusion does not hold in our uncompleted version of equivariant quantum cohomology; one has instead \eqref{eq:main-isomorphism} which is often non-zero even when $\mathrm{QH}^{*}(L)=0$.
\end{remark}

\begin{remark}
  This final result can be interpreted as an suitable quantum analogue of a result in algebraic topology relating $\Z/2\Z$- and $S^{1}$-equivariant cohomology, see, e.g., \cite[pp.\,957]{seidel-GAFA-2015} (in the same manner that \cite{biran-khanevsky-CMH-2013} is a quantum analogue of the Gysin sequence in algebraic topology).
\end{remark}

\begin{proof}[Proof of Theorem \ref{theorem:application} using Theorem \ref{theorem:main-computation}]
  Assume $n\ge 2$, as the case $n=1$ can be handled by the elementary argument given in \S\ref{sec:elem-argum-prov}. Then Theorem \ref{theorem:main-computation} applied with $L=\mathrm{pr}^{-1}(\bar{L})$ implies:
  \begin{equation*}
    \mathrm{QH}^{*}_{eq}(L)\simeq \mathrm{QH}^{*}(\bar{L})\oplus \mathrm{QH}^{*-1}(\bar{L})
  \end{equation*}
  Then Theorem \ref{theorem:displacement} implies $L$ is not equivariantly displacable (since the hypothesis of Theorem \ref{theorem:application} is that $\mathrm{QH}^{*}(\bar{L})\ne 0$)
\end{proof}

\subsection{The case of free actions}
\label{sec:case-free-actions}

Suppose the context for Theorem \ref{theorem:main-computation}, i.e., $\Omega$ is a Liouville domain, $\bd\Omega$ has a one-periodic Reeb flow, $\pi_{2}(\Omega,\bd\Omega)=0$, and the time-1/2 flow extends to $\Omega$ as an exact involution $a$. In this section, we are concerned with the case when $a$ acts freely on $\Omega$, and so one has a smooth quotient domain $\Omega'$. Then we can deduce:
\begin{theorem}\label{theorem:application-2}
  If a symplectic involution $a$ acts freely on a Liouville domain $\Omega$ satisfying $\pi_{2}(\Omega,\bd\Omega)=0$, with smooth quotient domain $\Omega'$, and if: $$\bar{L}\subset M=\bd\Omega/(\R/\Z)$$ is a monotone Lagrangian with minimal Maslov number at least $2$, then:
  \begin{equation*}
    \mathrm{QH}^{*}(L')\simeq \mathrm{QH}^{*}(\bar{L})\oplus \mathrm{QH}^{*-1}(\bar{L}),
  \end{equation*}
  where $L'$ is the lift of $\bar{L}$ to $\bd \Omega'$. Equivalently, the Floer-Euler class associated to the covering $L'\to \bar{L}$ vanishes.
\end{theorem}
What is interesting is the implication:
\begin{equation*}
  \mathrm{QH}^{*}(\bar{L})\ne 0\implies \mathrm{QH}^{*}(L')\ne 0,
\end{equation*}
as it enables one to bootstrap non-vanishing quantum cohomology in a lower dimension to a higher dimension. This result can be applied to, e.g., $T^{*}S^{n}$, $n\ge 3$, to produce monotone Lagrangians in $T^{*}\mathbb{R}P^{n}$ which are disjoint from the zero section and have non-vanishing quantum cohomology (by pulling back Lagrangians in the reduction of $ST^{*}S^{n}$, which is known to be the quadric hypersurface in $\mathbb{C}P^{n}$).

The proof of Theorem \ref{theorem:application-2} is based on a variation of a result of \cite{sampietro-christ-arXiv-2025}:
\begin{theorem}\label{theorem:free-action}
  If $W$ is a tame-at-infinity symplectic manifold with a symplectic involution $a$ which acts freely, and $L\subset W$ is a monotone Lagrangian with minimal Maslov number at least $2$, there is an isomorphism: $$\mathrm{QH}^{*}_{eq}(L)\simeq \mathrm{QH}^{*}(L'),$$ where $L'\subset W'$ is the quotient of $L\subset W$ by the involution.
\end{theorem}
\begin{proof}[Proof of Theorem \ref{theorem:application-2}]
  Using Theorem \ref{theorem:main-computation} and Theorem \ref{theorem:free-action} we have:
  \begin{equation*}
    \mathrm{QH}^{*}(L')\simeq \mathrm{QH}_{eq}^{*}(L)\simeq \mathrm{QH}^{*}(\bar{L})\oplus \mathrm{QH}^{*-1}(\bar{L}),
  \end{equation*}
  which gives the isomorphism. To see that the Floer-Euler class vanishes, we count dimensions using the Floer-Gysin sequence for $L'\to \bar{L}$ from \S\ref{sec:biran-khanevnsky-floer-gysin-sequence}:
  \begin{equation*}
    \begin{tikzcd}
      \cdots\arrow[r]&{\mathrm{QH}^{*+1}(L')}\arrow[r,"{A}"] &{\mathrm{QH}^{*}(\bar{L})}\arrow[r,"{F}"] &{\mathrm{QH}^{*+2}(\bar{L})}\arrow[r]&\cdots.
    \end{tikzcd}
  \end{equation*}
  We conclude from this sequence that $\dim \ker A=\dim \mathrm{QH}^{*+1}(\bar{L})$, which implies that $\dim \ker F=\dim \mathrm{QH}^{*}(\bar{L})$, and so multplication by $F$ is the zero map, i.e., $F=0$.
\end{proof}

\subsection{Further questions}
\label{sec:further-questions}

Does an analogous story hold with when $\Z/2\Z$ is replaced\footnote{Let us note that the linear symplectic group acts by $\Z/2\Z$-equivariant Hamiltonian isotopies of $\C^{n}$, but the action of a linear symplectic matrix is rarely $\Z/p\Z$-equivariant.} by another prime cyclic group $\Z/p\Z$? Surely the construction of the invariant and its main formal properties does not require $\Z/2\Z$, (see, e.g., \cite{shelukhin-zhao-JSG-2021} for a generalization of \cite{seidel-GAFA-2015} to other prime cyclic groups). However, it seems our computation of the equivariant cohomology vis-a-vis the circle bundle construction of \cite{biran-khanevsky-CMH-2013} requires further analysis in the case of $\Z/p\Z$ equivariant cohomology if $p>2$.

Another question: can one replace the assumption that $L$ is monotone with the weaker assumption that $\mu(A)>0$ for each class $A$ represented by a non-constant $J$-holomorphic disk? We use the stronger assumption of monotonicity in a seemingly crucial way in the proof that our (uncompleted) equivariant cohomology $\mathrm{QH}^{*}_{eq}$ is sensitive to equivariant displacements.

Another question: in the case of $\C^{n}$, can one prove similar equivariant non-displaceability of Lagrangians results using generating functions à la \cite{viterbo-mathann-1992}?

\subsection{Outline of the rest of the paper}
\label{sec:outline-rest-paper}

It remains to construct $\mathrm{QH}_{eq}^{*}$, and prove it is sensitive to equivariant displacements (Theorem \ref{theorem:displacement}); this is done in \S\ref{sec:equiv-quant-and-equiv-displ}. It also remains to prove Theorems \ref{theorem:main-computation} and \ref{theorem:free-action}; this is done in \S\ref{sec:proof-theorem-main-computation} and \S\ref{sec:proof-theorem-free-action}.

\subsection{Acknowledgements}
\label{sec:acknowledgements}

The first author thanks Daren Chen, Octav Cornea, Laurent Côte, Yakov Eliashberg, and Eric Kilgore for discussions on the topics of equivariant cohomology and pearl trajectories. This paper was written during the second author’s PhD, and he is indebted to his supervisor, Claude Viterbo, for invaluable support. This research was supported by funding from the ANR grant CoSy and the Université de Montréal DMS.

\section{Equivariant quantum cohomology and equivariant displacements}
\label{sec:equiv-quant-and-equiv-displ}

In \S\ref{sec:recoll-lagr-quant} we recall the non-equivariant theory in a bit more detail, and we discuss the special role of disks $u$ with $\mu(u)=2$ in \S\ref{sec:maslov-2-disks}. In \S\ref{sec:equiv-diff} we construct the equivariant differential $d_{eq}$ for regular data and prove that it squares to zero. In \S\ref{sec:cont-isom} we explain that the resulting invariant is independent of the choice of regular data as described in Remark \ref{remark:invariance-statement}. In \S\ref{sec:proof-theorem-displacement}, we show the vanishing of $\mathrm{QH}^{*}_{eq}$ in the presence of an equivariant Hamiltonian displacement.

\subsection{Recollection of Lagrangian quantum cohomology}
\label{sec:recoll-lagr-quant}

We continue the discussion started in \S\ref{sec:quant-cohom}. Let $\mathscr{G}_{k}$ be the (open) parameter space of ordered tuples $b=(b_{1},\dots,b_{k})$ satisfying $-\infty<b_{1}<\dots<b_{k}<\infty$, so $\mathscr{G}_{k}$ is a smooth open manifold of dimension $k$.

A \emph{pearl trajectory} is a tuple $w=(b;u_{1},\dots,u_{k})$ where $b\in \mathscr{G}_{k}$ and:
\begin{equation*}
  \left\{\begin{aligned}
    &u_{\nu}:\R\times [0,1]\to (W,L),\\
    &\bd_{s}u_{\nu}+J(u_{\nu})\bd_{t}u_{\nu}=0,\\
    &0<\textstyle\int \omega(\bd_{s}u_{\nu},\bd_{t}u_{\nu})dsdt<\infty,
  \end{aligned}\right.
\end{equation*}
for $\nu=1,\dots,k$, such that the flow line of $P$ of time $b_{\nu+1}-b_{\nu}$ joins the removable singularities $u_{\nu}(+\infty)$ to $u_{\nu+1}(-\infty)$, $\nu=1,\dots,k-1$. Each pearl trajectory has asymptotic zeros $x_{-},x_{+}$ of $P$ obtained by flowing $u_{1}(-\infty)$ backwards by $P$ and $u_{k}(+\infty)$ forwards by $P$. In other words, $u_{1}(-\infty)$ lies in the stable manifold $S_{x_{-}}$, and $u_{k}(+\infty)$ lies in the unstable manifold $U_{x_{+}}$.

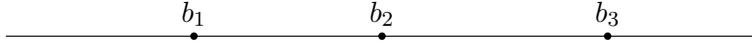
\begin{figure}[h]
  \centering
  \begin{tikzpicture}
    \draw[every node/.style={fill,circle,inner sep=1pt}] (0,0)--node[pos=0.25](A){}node[pos=0.5](B){}node[pos=0.8](C){}(10,0);
    \path (A)node[above]{$b_{1}$}--(B)node[above]{$b_{2}$}--(C)node[above]{$b_{3}$};
  \end{tikzpicture}
  \caption{$\mathscr{G}_k$ is the parameter space of strictly ordered $k$-tuples (shown with $k=3$).}
  \label{fig:parameter-space-Gk}
\end{figure}
\begin{figure}[h]
  \centering
  \begin{tikzpicture}

    \draw[every node/.style={fill=black,circle,inner sep=1pt}] (0,0)--node[pos=0.25](A){}node[pos=0.5](B){}node[pos=0.8](C){}(10,0);
    \draw[every node/.style={fill=white,draw,circle,inner sep=8pt}] (A)node{} (B)node{} (C)node{};
    \node at (0,0) [left]{$x_{-}$};
    \node at (10,0) [right]{$x_{+}$};
  \end{tikzpicture}
  \caption{Illustration of a pearl trajectory $w$ with $k=3$ and asymptotics $x_{-},x_{+}$.}
  \label{fig:parameter-space-Gk}
\end{figure}
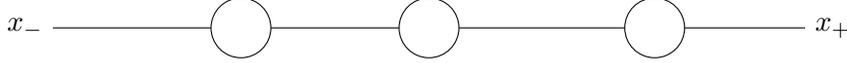

For generic choice of $J$, the space of pearl trajectories $w$ with $k$ pearls and $\mu(w)=\sum \mu(u_{\nu})=Ni$ is a smooth manifold of dimension:
\begin{equation}
  \label{eq:dimension-formula}
  Ni+\mathrm{ind}(x_{+})-\mathrm{ind}(x_{-})+k;
\end{equation}
moreover, the space of pearl trajectories carries a free $\R\times \R^{k}$ action, where the $\R$ action translates $g$, and the $\R^{k}$ action translates the ``pearls'' $u_{1},\dots,u_{k}$.

Still assuming a generic $J$, the space of rigid orbits of this $\R^{k+1}$-action is a compact zero dimensional manifold. The count of these rigid orbits with asymptotics $x_{-},x_{+}$ produces a coefficient $N_{i,k}(x_{-},x_{+})\in \mathbf{k}$, and one defines:
\begin{equation*}
  d_{i}(x_{-}):=\sum_{k=0}^{\infty} N_{i,k}(x_{-},x_{+})x_{+}.
\end{equation*}
Finally, we observe the rigid orbits satisfy $\mathrm{ind}(x_{+})-\mathrm{ind}(x_{-})=1-Ni$; thus $d_{i}$ defines an endomorphism of degree $1-Ni$.

\begin{theorem}[Biran-Cornea]
  For regular data $(P,J)$, the differential: $$d=\sum q^{i}d_{i}$$ squares to zero. (See Remark \ref{remark:non-eq-regular} for the definition of regular).
\end{theorem}

This result is a special case of our result that the equivariant differential squares to zero in \S\ref{sec:equiv-diff}. One subtlety $v$ concerns disks with $\mu(v)=2$, discussed next.

\subsection{Disks with Maslov number 2}
\label{sec:maslov-2-disks}

The definition of Lagrangian quantum cohomology depends on a choice of $(P,J)$. It will be important to consider:
\begin{itemize}
\item the moduli space $\mathscr{M}_{2}(L,J)$ of $J$-holomorphic disks $v$ with $\mu(v)=2$,
\item the moduli space $\mathscr{M}_{1}(J)$ of $J$-holomorphic spheres with $c_{1}(u)=1$.
\end{itemize}
Such special treatment of disks with Maslov number 2 appears throughout the literature on enumerating holomorphic disks with boundary on a Lagrangian; we refer the reader to \cite[\S3.3]{biran-cornea-GT-2012}, and \cite{oh-addendum-CPAM-1995} for further discussion.

Each $v\in \mathscr{M}_{2}(L,J)$ or $v\in \mathscr{M}_{1}(J)$ determines a linearized operator: $$D_{v}:W^{1,p}(v^{*}TW)\to L^{p}(v^{*}TW)$$
which is obtained by differentiating the $J$-holomorphic curve equation.
\begin{lemma}\label{lemma:xversality-maslov-2}
  Fix a precompact neighborhood $U$ of $L$ and an $\omega$-tame almost complex structure $J_{0}$. Let $\mathscr{J}(U)$ be the set of $\omega$-tame almost complex structures which agree with $J_{0}$ outside of $U$, topologized with the $C^{\infty}$ topology. There is an open dense set $\mathscr{J}^{\mathrm{reg}}(U)\subset \mathscr{J}(U)$ of complex structures $J$ which render $D_{v}$ surjective for each $v\in \mathscr{M}_{2}(L,J)$ and for each $v\in \mathscr{M}_{1}(J)$ whose image intersects $U$.
\end{lemma}
\begin{proof}
  It is straightforward to adapt the arguments of \cite{mcduff-salamon-book-2012} to handle those $v$ which have ``injective points'' mapped into $U$. Each sphere $v$ with Chern number $1$ cannot be a multiple cover, and so $v$ must simple\footnote{\emph{simple}, in the sense that $v$ is an injective holomorphic immersion on the complement of a finite set of critical points and self-intersections.} by the work of \cite{micallef-white}. It follows that if $v$ intersects $U$ then $v$ has injective points mapped into $U$. In the case of holomorphic disks with Maslov number $2$, we follow \cite[\S3.2]{biran-cornea-arXiv-2007} to explain why we can assume such injective points exist. The key result \cite[Theorem 3.2.1]{biran-cornea-arXiv-2007} is due to \cite{lazzarini-1,lazzarini-2} which proves that any holomorphic disk $v$ can be decomposed into multiple covers of simple disks. If $v$ has the minimal Maslov number $2$, it must therefore have injective points passing through $U$. This completes the proof.
\end{proof}

To each $J\in \mathscr{J}^{\mathrm{reg}}(U)$ we can then define:
\begin{equation*}
  \begin{aligned}
    \Pi_{2}&=(\mathscr{M}_{2}(L,J)\times \bd D(1)\times \bd D(1))/\mathrm{Aut}(D(1)),\\
    \Pi_{1}&=\mathrm{ev}_{1}^{-1}(U)\text{ where }\mathrm{ev}_{1}:(\mathscr{M}_{1}(J)\times S^{2})/\mathrm{Aut}(S^{2})\to W,
  \end{aligned}
\end{equation*}
with evaluation maps $\mathrm{ev}_{2}:\Pi_{2}\to L\times L$ and $\mathrm{ev}_{1}:\Pi_{1}\to W$. By the regularity ensured by Lemma \ref{lemma:xversality-maslov-2}, $\Pi_{2}$ is a smooth manifold of dimension $n+1$ and $\Pi_{1}$ is a smooth manifold of dimension $2n-2$.

\begin{lemma}
  Given a Morse-Smale pseudogradient $P$, there is an open and dense subset $\mathscr{J}^{\mathrm{reg}}(P,U)\subset \mathscr{J}^{\mathrm{reg}}(U)$ such that:
  \begin{enumerate}
  \item $\mathrm{ev}_{1}^{-1}(x)=\emptyset$ for each zero $x$ of $P$,
  \item $\mathrm{ev}_{2}$ is transverse to $S_{x}\times U_{y}$ for each pair of zeros $x,y$ of $P$, where $S_{x}$ and $U_{y}$ are the stable and unstable manifolds, respectively.
  \end{enumerate}
\end{lemma}
\begin{proof}
  The conditions are open conditions, so only needs to prove their density; this in turn follows from standard application of the Sard-Smale theorem as in \cite{mcduff-salamon-book-2012}, together with the injectivity considerations from Lemma~\ref{lemma:xversality-maslov-2}. The first condition is equivalent to $\mathrm{ev}_{1}$ being transverse to each singleton $\set{x}$, since $\dim\Pi_{1}<\mathrm{codim}\set{x}$.
\end{proof}

For $J\in \mathscr{J}^{\mathrm{reg}}(P,U)$ it then follows that:
\begin{equation*}
  \Pi_{2}^{x}:=\mathrm{ev}^{-1}_{2}(S_{x}\times U_{x})
\end{equation*}
is an open $1$-manifold whose non-compact ends are in bijection with the zero manifold:
\begin{equation*}
  \bd \Pi_{2}^{x}=\mathrm{ev}^{-1}_{2}(\bd S_{x}\times U_{x})\cup \mathrm{ev}_{2}^{-1}(S_{x}\times \bd U_{x}),
\end{equation*}
where $\bd S_{x}$ and $\bd U_{x}$ are the codimension $1$ boundaries. This observation is used to prove that $d^{2}=0$ in the case $N=2$.

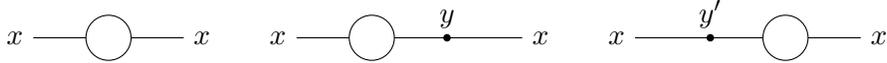
\begin{figure}[h]
  \centering
  \begin{tikzpicture}
    \begin{scope}
      \draw (0,0)node[left]{$x$}--(2,0) node[right]{$x$};
      \draw[fill=white] (1,0) circle (0.3);
    \end{scope}
    \begin{scope}[shift={(3.5,0)}]
      \draw (0,0)node[left]{$x$}--(3,0) node[right]{$x$};
      \node[fill,inner sep=1pt,circle] (X) at (2,0) {};
      \node at (X)[above]{$y$};
      \draw[fill=white] (1,0) circle (0.3);
    \end{scope}
    \begin{scope}[shift={(8,0)}]
      \draw (0,0)node[left]{$x$}--(3,0) node[right]{$x$};
      \node[fill,inner sep=1pt,circle] (X) at (1,0) {};
      \node at (X)[above]{$y'$};
      \draw[fill=white] (2,0) circle (0.3);
    \end{scope}
  \end{tikzpicture}
  \caption{The 1-manifold $\Pi^{x}_{2}$ and example boundary configurations. Here $y$ is in the unstable manifold of $x$ and $y'$ is in the stable manifold of $x$.}
  \label{fig:1-manifold-Pi-x-2}
\end{figure}

\subsection{The equivariant differential}
\label{sec:equiv-diff}

We continue from \S\ref{sec:equiv-quant-cohom}. The goal is to define the moduli spaces used in the definition of the equivariant differential, and define a notion of \emph{regular data}; such data will yield a well-defined equivariant differential which squares to zero.

\subsubsection{Borel data}
\label{sec:borel-data}

The terms in the equivariant differential are defined using data parametrized by: $$S^{\infty}=\set{(\eta_{0},\eta_{1},\dots):\sum \eta_{i}^{2}=1\text{ and }\eta_{i}=0\text{ for }i\text{ sufficiently large}}.$$

The required data is:
\begin{enumerate}
\item\label{item:eq-data-1} a family $P_{\eta}$ of vector fields on $L$, $\eta\in S^{\infty}$,
\item\label{item:eq-data-2} a family $J_{\eta}$ of $\omega$-tame almost complex structures, $\eta\in S^{\infty}$,
\end{enumerate}
and this so-called \emph{Borel data} is required to satisfy the following properties:
\begin{enumerate}[resume]
\item\label{item:eq-data-3} $(P_{-\eta},J_{-\eta})=(a^{*}P_{\eta},a^{*}J_{\eta})$,
\item\label{item:eq-data-4} $(P_{\eta},J_{\eta})=(P,J)$ if $\eta_{j}\ge 4/5$,
\item\label{item:eq-data-5} $(P_{\tau(\eta)},J_{\tau(\eta)})=(P_{\eta},J_{\eta})$,
\end{enumerate}
where $\tau:S^{\infty}\to S^{\infty}$ is the self-similarity shift map: $$\tau(\eta_{0},\eta_{1},\dots)=(0,\eta_{0},\eta_{1},\dots).$$
It is important to note that the neighborhoods $\set{\eta_{j}\ge 4/5}$ and $\set{\eta_{j}\le -4/5}$ are mutually disjoint neighborhoods of the ``poles:''
\begin{equation*}
  v_{j,\pm}:=(0,\dots,0,\pm 1,0,\dots)\text{ where the $\pm 1$ is in the $j$th position.}
\end{equation*}
As explained in \cite{seidel-smith-GAFA-2010} and \cite[\S3]{seidel-GAFA-2015}, there is a vector field $\mathfrak{S}$ on $S^{\infty}$ which behaves as a Morse-Smale pseudogradient for the function $\sum j\eta_{j}^{2}$. To be precise, we will take $\mathfrak{S}$ as the generator of the flow:
\begin{equation*}
  [\eta_{0}:\eta_{1}:\eta_{2}:\dots]\mapsto [\eta_{0}:e^{t}\eta_{1}:e^{2t}\eta_{2}:\dots],
\end{equation*}
defined on $\mathrm{RP}^{\infty}$ using real projective coordinates (this lifts uniquely to a pseudogradient on $S^{\infty}$). Let us note that:
\begin{itemize}
\item the poles $v_{j}^{\pm}$ are the zeros of $\mathfrak{S}$,
\item $\mathfrak{S}(-z)=-\mathfrak{S}(z)$, using the identification of $TS^{\infty}_{z}$ and $TS^{\infty}_{-z}$,
\item $\mathfrak{S}$ is tangent to $S^{n}\subset S^{\infty}$,
\item $\mathfrak{S}(\tau(z))=\tau_{*}\mathfrak{S}(z)$.
\end{itemize}
Associated to this pseudogradient is the space of parametrized flow lines $\mathscr{P}^{j,\pm}$ asymptotic to $v_{0,+}$ and $v_{j,\pm}$, which forms an (open) manifold of dimension $j$, and carries an $\R$-action by translation.

\subsubsection{Definition of the differential}
\label{sec:defin-diff}

The terms $d_{i,j}$ are of the form:
\begin{equation*}
  d_{i,j}=\sum_{k=0}^{\infty}d_{i,j,k}^{+}+d_{i,j,k}^{-},
\end{equation*}
where $d_{i,j,k}^{\pm}$ is defined by counting pearl trajectories in the parametric moduli space of solutions $w=(b,\pi,u_{1},\dots,u_{k})$ where:
\begin{equation*}
  \left\{
    \begin{aligned}
      &b=(b_{1},\dots,b_{k})\in \mathscr{G}_{k}\text{ and }\pi\in \mathscr{P}^{j,+},\\
      &\bd_{s}u_{\nu}+J_{\pi(b_{\nu})}\bd_{t}u_{\nu}=0\text{ and }0<\omega(u_{\nu})<\infty,\text{ for }\nu=1,\dots,k,\\
      &\mu(u_{1})+\dots+\mu(u_{k})=Ni,
    \end{aligned}
  \right.
\end{equation*}
such that the flow line of $P_{\pi(s)}$, $s\in [b_{\nu},b_{\nu+1}]$ joins $u_{\nu}(+\infty)$ to $u_{\nu+1}(-\infty)$.

Let us denote this moduli space by $\mathscr{M}_{i,j,k}^{\pm}$. As above, this has an $\R^{k+1}$ action, and we are interested in the rigid orbits.

In the ``$+$'' case, such a solution has asymptotic zeros $x_{-}(w),x_{+}(w)$ of $P$ at its ends, by flowing $b_{1}(-\infty)$ backwards by $P_{\pi(s)}$, $s\in (-\infty,b_{1}]$, and by flowing $b_{k}(+\infty)$ forwards by $P_{\pi(s)}$, $s\in [b_{k},\infty)$. Define:
\begin{equation*}
  N_{i,j,k}^{+}(x_{-},x_{+})=\text{count of rigid orbits of $\mathscr{M}^{+}_{i,j,k}$ with asymptotics $x_{-},x_{+}$}.
\end{equation*}

In the ``$-$'' case, one can still associate zeros $x_{-},x_{+}$ of $P$, where $x_{-}$ is as above, but now $a(x_{+})$ (instead of $x_{+}$) is the asymptotic of the trajectory of $P_{\pi(s)}$, $s\in [b_{k},\infty)$, starting at $u_{k}(+\infty)$. In words, the ``$-$'' case twists the geometric asymptotic by the involution $a$. With this convention, each solution $w$ has asymptotic zeroes $x_{\pm}(w)$ of $P$. Define:
\begin{equation*}
  N_{i,j,k}^{-}(x_{-},x_{+}):=\text{count of rigid orbits of $\mathscr{M}^{-}_{i,j,k}$ with asymptotics $x_{-},x_{+}$}.
\end{equation*}
This twisting of the asymptotics is the same as in \cite{seidel-smith-GAFA-2010,seidel-GAFA-2015}.

Our main structural result is the following:
\begin{theorem}\label{theorem:diff-structural}
  If $(P,J)$ are regular for defining the quantum comology, then they can be extended to Borel data $(P_{\eta},J_{\eta})$ satisfying \ref{item:eq-data-1}--\ref{item:eq-data-5} which is regular in the sense of Definition \ref{definition:regular} below. In this case, the endomorphisms of $\mathrm{CM}(P)$ given by:
  \begin{equation*}
    d_{i,j,k}^{\pm}(x_{-})=\sum_{x_{+}}N_{i,j,k}^{\pm}(x_{-},x_{+})x_{+}
  \end{equation*}
  have degree $1-j-Ni$ and $d_{eq}=\sum e^{j}q^{i}d_{i,j,k}^{\pm}$ squares to zero; the specialization to $j=0$ agrees with the non-equivariant quantum differential.
\end{theorem}

The precise notion of regularity is given in Definition \ref{definition:regular-b-data}, and Theorem \ref{theorem:diff-structural} is proved in \S\ref{sec:proof-theorem-diff-structural}.

\subsubsection{Linearized operators}
\label{sec:linearized-operators}

In this section we consider classes of Borel data which achieve transversality for various moduli spaces.

Let $w=(b,\pi,u_{1},\dots,u_{k})$ be a pearl trajectory with $k$ pearls, total Maslov number $Ni$, and with $\pi\in \mathscr{P}=\mathscr{P}^{j,\pm}$, using Borel data $(P_{\eta},J_{\eta})$. Here we assume that the data $(P,J)$ near the positive poles is generic in the sense of \S\ref{sec:maslov-2-disks} if the minimal Maslov number is $2$.

We define the \emph{ambient tangent space} of $w$ to be:
\begin{equation*}
  AT_{w}=W^{1,p,\delta}(u_{1}^{*}TW)\oplus \dots\oplus W^{1,p,\delta}(u_{k}^{*}TW)\oplus T\mathscr{G}_{k,b}\oplus T\mathscr{P}_{\pi}\oplus TL^{\oplus 2k};
\end{equation*}
where we abbreviate $W^{1,p,\delta}(u_{\nu}^{*}TW)=W^{1,p,\delta}(u_{\nu}^{*}TW,u_{\nu}^{*}TL)$  and:
\begin{equation*}
  TL^{\oplus 2k}=TL_{p_{1,-}}\oplus TL_{p_{1,+\infty}}\oplus \dots \oplus TL_{p_{k,-}}\oplus TL_{p_{k,+}},
\end{equation*}
with $p_{i,\pm}=u_{i}(\pm \infty)$. This is the tangent space of all variations of $w$, including those which do not respect the incidence conditions.

The usual Floer theory linearization process produces a linear operator:
\begin{equation*}
  AD_{w}:AT_{w}\to L^{p,\delta}(u_{1}^{*}TW)\oplus \dots \oplus L^{p,\delta}(u_{k}^{*}TW);
\end{equation*}
which maps $W^{1,p,\delta}(u_{i}^{*}TW)$ into $L^{p,\delta}(u_{i}^{*}TW)$ as an operator of Cauchy-Riemann type (in particular, as a Fredholm operator). For related details, we refer the reader to \cite[\S4]{cant-thesis-2022}, \cite[\S4.1.2]{brocic-cant-JFPTA-2024}, and \cite[\S4.4]{biran-cornea-arXiv-2007}.

To incorporate the incidence conditions, we define:
\begin{equation*}
  \mathscr{I}\subset T\mathscr{G}_{k,b}\oplus TP_{\pi}\oplus TL^{\oplus 2k}
\end{equation*}
to be the subspace of \emph{coincident variations}: to state the definition, let us denote by $\phi_{b_{1},b_{2}}^{\pi}$ the flow by $P_{\pi(s)}$ on the interval $[b_{1},b_{2}]$, where $b_{1}<b_{2}$. For each $i=1,\dots,k-1$, differentiating:
\begin{equation*}
  \phi^{\pi}_{b_{i},b_{i+1}}(p_{i,+})=p_{i+1,-},
\end{equation*}
gives a linear map:
\begin{equation*}
  C_{i}:T\mathscr{G}_{b}\oplus T\mathscr{P}_{\pi}\oplus TL_{p_{i,+}}\to TL_{p_{i+1,-}}.
\end{equation*}
Similarly one can differentiate the relation that $\phi^{\pi}_{-\infty,b_{1}}(p_{1,-})$ lies in the stable manifold of $x_{-}$ and the relation that $\phi^{\pi}_{b_{k},+\infty}(p_{k,+})$ lies in the unstable manifold of $x_{+}$ as follows: flowing for a sufficiently long time and then projecting onto unstable and stable spaces produces maps:
\begin{equation*}
  \begin{aligned}
    &C_{0}:T\mathscr{G}_{k,b}\oplus T\mathscr{P}_{\pi}\oplus TL_{p_{1,-}}\to U_{x_{-}},\\
    &C_{k}:T\mathscr{G}_{k,b}\oplus T\mathscr{P}_{\pi}\oplus TL_{p_{k,+}}\to S_{x_{+}}.
  \end{aligned}
\end{equation*}
In order to lie in the stable/unstable manifold, $C_{0}$ and $C_{k}$ should vanish. We explain these maps with a bit more detail: by definition of Morse-type, $P$ has distinguished linear charts around its critical points (where the flow is diagonal and linear with no zero eigenvalues). This linear chart gives a decomposition:
\begin{equation*}
  \mathrm{pr}_{U}\times \mathrm{pr}_{S}:\mathrm{Nbd}(x)\to U_{x}\times S_{x},
\end{equation*}
where $U_{x}$ is the direct sum of the negative eigenspaces and $S_{x}$ is the direct sum of the positive eigenvalues. By differentiating:
\begin{equation*}
  \mathrm{pr}_{S}\circ \phi_{b_{k},\tau}^{\pi}(p_{k,+})\in S_{x_{+}},
\end{equation*}
for $\tau$ sufficiently large, one obtains a map $C_{k}$ which is well-defined up to composition with an isomorphism. In particular, the kernel of $C_{k}$ is well-defined independently of $\tau$. The kernel of $C_{0}$ is defined similarly.

The variation $(\delta b,\delta \pi, \delta p)$ lies in $\mathscr{I}$ provided:
\begin{enumerate}
\item $\delta p_{i+1,-}=C_{i}(\delta b, \delta \pi,\delta p_{i,+})$, and,
\item $C_{0}(\delta b,\delta \pi, \delta p_{1,-})=0$ and $C_{k}(\delta b,\delta \pi,\delta p_{k,+})=0$.
\end{enumerate}
A moment's thought reveals that $\mathscr{I}$ is cut out transversally and has codimension: $$n(k-1)+\dim TU_{x_{-}}+\dim TS_{x_{+}}.$$
It then follows that the restriction $D_{w}=AD_{w}|_{T_{w}}$ where:
\begin{equation*}
  T_{w}=W^{1,p,\delta}(u_{1}^{*}TW)\oplus \dots \oplus W^{1,p,\delta}(u_{k}^{*}TW)\oplus \mathscr{I}
\end{equation*}
is a Fredholm operator of index:
\begin{equation*}
  \mathrm{index}(D_{w})=\mu(w)+\dim TU_{x_{+}}-\dim TU_{x_{-}}+k+\dim T\mathscr{P}_{\pi}.
\end{equation*}
This leads to the definition of a regular solution:
\begin{definition}\label{definition:regular}
  A solution $w$ is regular if $\dim \ker (D_{w})=\mathrm{index}(D_{w})$.
\end{definition}

It is convenient to introduce $\mathrm{ind}(x)=\dim TU_{x}$ and the \emph{virtual dimension}:
\begin{equation}\label{eq:virtual-dimension-differential}
  \mathrm{vdim}(w)=\mu(w)+\mathrm{ind}(x_{+})-\mathrm{ind}(x_{-})+\dim T\mathscr{P}_{\pi}-1.
\end{equation}
This equals $\mathrm{index}(D_{w})-k-1$, and is the expected dimension of $[w]$ in the quotient by the $\R^{k+1}$ action which acts by translations on the pearls (this gives an $\R^{k}$ action) and by simultaneous translation of $b\in \mathscr{G}$ and $\pi\in \mathscr{P}$ (this gives an additional $\R$ action).

\subsubsection{Nodal curves}
\label{sec:nodal-curv}

Unfortunately, to prove finiteness of the counts appearing in $d_{eq}$, and to prove $d_{eq}^{2}=0$ in the standard way, it will be necessary to enforce additional transversality requirements for certain nodal configuations. These curves can be considered as lying above the compactification of $\mathscr{G}_{k}$.

Given a surjective non-decreasing map $f:\set{1,\dots,k}\to \set{1,\dots,\ell}$, consider solutions $w=(b,\pi,u_{1},\dots,u_{k})$ of the \emph{nodal equation}:
\begin{equation*}
  \left\{
    \begin{aligned}
      &u_{i}:\R\times [0,1]\to (W,L),\\
      &\bd_{s}u_{i}+J_{\pi(b_{f(i)})}\bd_{t}u_{i}=0,\\
      &0<\omega(u_{i})<\infty,
    \end{aligned}
  \right.
\end{equation*}
subject to the incidence equation:
\begin{equation*}
  \phi^{\pi}_{b_{f(i)},b_{f(i+1)}}(u_{i}(+\infty))=u_{i+1}(-\infty),
\end{equation*}
where the flow $\phi$ is defined above. This incidence forces $u_{i}(+\infty)=u_{i+1}(-\infty)$ whenever $f(i)=f(i+1)$. See Figure \ref{fig:interface-curve} for an illustration.

Similarly to \S\ref{sec:linearized-operators}, such curves also have a linearized operator $D_{w}$ defined on a space $T_{w}=W^{1,p,\delta}\oplus \mathscr{I}$ of coincident variations with index:
\begin{equation*}
  \mathrm{index}(w)=\mu(w)+\mathrm{ind}(x_{+})-\mathrm{ind}(x_{-})+\ell+\dim \mathscr{P}.
\end{equation*}
This linearized operator has a finite dimensional kernel (consisting of smooth coincident variations).
One can quotient the kernel by an $\R^{k+1}$ action; the expected dimension of this quotient space equals the \emph{virtual dimension}:
\begin{equation}\label{eq:nodal-vdem}
  \mathrm{vdim}(w)=\mu(w)+\mathrm{ind}(x_{+})-\mathrm{ind}(x_{-})+\dim \mathscr{P}-1-\mathfrak{n},
\end{equation}
(where $\mathfrak{n}=k-\ell$ is the number of nodes). Then:
\begin{definition}\label{definition:regular-b-data}
  Borel data $(P_{\eta},J_{\eta})$ is called regular if $D_{w}$ is surjective for each (nodal) solution $w$ satisfying $\mathrm{vdim}(w)\le -\mathfrak{n}$,  and if $(P,J)$ is regular in the sense of \S\ref{sec:maslov-2-disks}. This includes the case of non-nodal solutions with $\mathfrak{n}=0$.
\end{definition}

For regular data, the moduli spaces of \emph{free} orbits of the $\R^{k+1}$-action with $\mathrm{vdim}=i$, have the structure of a smooth $i$-dimensional manifold, for $i\le 1$; in particular, \emph{non-constant} solutions with negative virtual dimension do not exist. However, the constant solutions $k=0$, $j=0$, $i=0$ do exist in the $\mathrm{vdim}=-1$ component (these are not free orbits of the $\R^{k+1}=\R^{1}$ action).

As explained in \S\ref{sec:defin-diff}, the solutions with $\mathrm{vdim}=\mathfrak{n}=0$ contribute to the definition of the Floer differential. For regular Borel data, the counts converge: by monotonicity, one has a priori upper bounds on $\dim \mathscr{P}$ and $\mu(w)$ and so compactness of the space of orbits with $\mathrm{vdim}(w)=0$ follows from the non-existence of non-constant solutions with negative virtual dimension (and also the non-existence of nodal solutions with negative virtual dimension).

\begin{definition}\label{definition:very-regular-b-data}
  We say that $(P_{\eta},J_{\eta})$ is very regular Borel data if it is regular and if additionally all solutions $w$ with $\mathrm{vdim}(w)\le 1-\mathfrak{n}$ are regular. This includes the nodal and non-nodal solutions.
\end{definition}

\begin{theorem}\label{theorem:regular-b-data}
  Given any set-up $(W,L,a)$ as above, there exists very regular Borel data $(P_{\eta},J_{\eta})$ in the sense of Definition \ref{definition:very-regular-b-data}, and such data can be chosen to be arbitrarily close to any chosen Borel data (i.e., very regular data is generic). If the involution $a$ acts freely on $W$, then one can take $(P_{\eta},J_{\eta})$ to be independent of $\eta$. In particular, (very) regular data for the non-equivariant quantum cohomology (see Remark \ref{remark:non-eq-regular}) also exists.
\end{theorem}
We postpone the proof to \S\ref{sec:proof-theorem-regular-b-data}; the argument is based on the one in \cite{biran-cornea-arXiv-2007}.

\begin{remark}\label{remark:non-eq-regular}
  It will be important to also mention when data $(P,J)$ for defining the \emph{non-equivariant} quantum cohomology is (very) regular. As a shortcut, we adopt the following definition: we consider $(P,J)$ as Borel data on the space $W\sqcup W$ where the involution interchanges the two factors (in this case $(P_{\eta},J_{\eta})$ is equal to $(P,J)$ on each factor and is independent of $\eta$). If this trivial extension is (very) regular Borel data, then we say $(P,J)$ is \emph{(very) regular for defining the non-equivariant quantum homology}.
\end{remark}

\begin{figure}[h]
  \centering
  \begin{tikzpicture}
    \begin{scope}
      \draw (0,0)--(3,0);
      \draw[fill=white] (1,0) circle (0.3) (2,0) circle (0.3);
    \end{scope}
    \begin{scope}[shift={(4,0)}]
      \draw (0,0)--(3,0);
      \draw[fill=white] (1.2,0) circle (0.3) (1.8,0) circle (0.3);
    \end{scope}
    \begin{scope}[shift={(8,0)}]
      \draw (0,0)--(3,0);
      \draw[fill=white] (1.5,0) circle (0.5);
    \end{scope}
  \end{tikzpicture}
  \caption{Nodal curve lying at the interface of two moduli spaces of non-nodal curves.}
  \label{fig:interface-curve}
\end{figure}
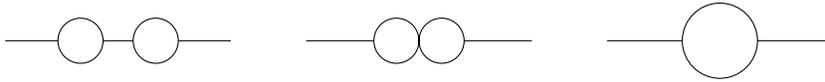

We continue to discuss the role of nodal curves in the proof that $d_{eq}^{2}=0$. The ideas are essentially the same as those used in \cite{cornea-lalonde-arXiv-2005,biran-cornea-arXiv-2007,biran-cornea-GT-2009}. The set of curves solving the nodal equation can be considered as compactifying the one-dimensional components in the moduli spaces of non-nodal curves, in a loose sense: a sequence of curves $w_{n}$ lying above $\mathscr{G}_{k}$ can fail to converge because $b_{n}$ diverges in $\mathscr{G}_{k}$ because some markers collide, say, $\lim b_{n,i}=b_{f(i)}$ for some map $f$ satisfying the requirements of the nodal equation. Ignoring for the moment other possible degenerations, $w_{n}$ would then converge to a nodal curve with $k>\ell$. By inspection, $\mathrm{vdim}(w)=\mathrm{vdim}(w_{n})-(k-\ell)$. Thus the virtual dimension drops by the number of nodal interfaces created.

There is another way to create nodal interfaces: one can have a sequence of curves $w_{n}$ which fails to converge because some of the $u_{n,i}$ break into a chain of holomorphic disks. Again ignoring the other failures of compactness, one concludes that $w_{n}$ converges to a nodal curve $w$ with $k>\ell$, and we still have $\mathrm{vdim}(w)=\mathrm{vdim}(w_{n})-(k-\ell)$.

As we explained above, the equivariant differential $d_{eq}$ is defined by counting non-nodal solutions with $\mathrm{vdim}(w)=0$; all of these solutions are non-nodal. To prove $d_{eq}$ squares to zero one needs to consider:
\begin{itemize}
\item $\mathscr{M}_{1}$ the non-nodal pearl trajectories with $\mathrm{vdim}(w)=1$,
\item $\mathscr{M}_{0,\mathrm{node}}$ nodal curves with $k=\ell+1$ (i.e., only one node) and $\mathrm{vdim}(w)=0$.
\end{itemize}
The idea is to glue together different components of $\mathscr{M}_{1}$ along $\mathscr{M}_{0,\mathrm{node}}$.

\subsubsection{Fibre product curves}
\label{sec:fibre-prod-curves}

To complete the proof that $d_{eq}^{2}=0$, we consider one final configuration of pearl trajectories. A \emph{fibre product} curve is a tuple of nodal solutions $w=(w_{1},\dots,w_{m})$, subject to the following conditions:
\begin{enumerate}
\item $x_{+}(w_{i})=x_{-}(w_{i+1})$, \emph{(in the twisted sense defined above)},
\item $\mathrm{vdim}(w_{i})\ge 0$.
\end{enumerate}
Set $x_{+}(w)=x_{+}(w_{m})$ and $x_{-}(w)=x_{-}(w_{1})$. We define the virtual dimension of a fibre product curve as the sum of $\mathrm{vdim}(w_{i})$. Let us call the interface between $w_{i}$ and $w_{i+1}$ a \emph{Morse breaking}.

Denote by $\mathscr{M}_{\mathrm{break}}(x,y)$ the collection of fibre product curves with $x_{-}(w)=x$ and $x_{+}(w)=y$, only one Morse breaking, and virtual dimension $0$. These appear as ends of the 1-dimensional component $\mathscr{M}_{1}(x,y)$.

\subsubsection{Proof of Theorem \ref{theorem:diff-structural}}
\label{sec:proof-theorem-diff-structural}

First we prove it for very regular data following the ideas of \cite{biran-cornea-arXiv-2007}, and the ideas of \cite{seidel-smith-GAFA-2010,seidel-GAFA-2015} for the $\pi\in \mathscr{P}$ component; one shows:
\begin{equation*}
  \# \bd \mathscr{M}_{1}(x,y)+2\# \mathscr{M}_{\mathrm{node}}(x,y)+\#\mathscr{M}_{\mathrm{break}}(x,y)=0\text{ mod }2
\end{equation*}
because each element of $\mathscr{M}_{\mathrm{node}}(x,y)$ can be glued into two different ways to produce ends of moduli spaces in $\mathscr{M}_{1}$, while each Morse breaking can only be glued in one way. Moreover, the count of $\mathscr{M}_{\mathrm{break}}$ represents the coefficients appearing in $d^{2}_{eq}$, and one thereby concludes that $d^{2}_{eq}=0$ modulo two.

The trick to upgrade from very regular (Definition \ref{definition:very-regular-b-data}) to just regular (Definition \ref{definition:regular}) is the following: one uses that regular data can be approximated by a converging sequence of very regular data (by Theorem \ref{theorem:regular-b-data}). Then, because we begin with regular data, the counts used in the definition of $d_{eq}$ stabilize in this approximating sequence. Thus one can replace the regular data by very regular data without changing $d_{eq}$, and the desired result follows. For a similar argument, see \cite[\S6.3]{biran-khanevsky-CMH-2013}. \hfill$\square$

In the remaining sections \S\ref{sec:gradient-bound}, \S\ref{sec:gradient-bound-exceptional}, \S\ref{sec:grom-floer-comp}, and \S\ref{sec:gluing} we comment on some of the technical points of the compactness and gluing results implicitly used in the above proof. The reader who is comfortable with the above proof that $d_{eq}^{2}=0$ is encouraged to continue to \S\ref{sec:cont-isom}.

\subsubsection{Gradient bound}
\label{sec:gradient-bound}

Consider a sequence of solutions $w_{n}$ for regular Borel data satisfying $\mathrm{vdim}(w_{n})\le 1$. Then:
\begin{lemma}[A priori $C^{1}$ bound]
  If $x_{-}(w_{n})\ne x_{+}(w_{n})$ and $\mathrm{vdim}(w_{n})\le 1$, then the first derivatives of each pearl $u_{n,i}$ are uniformly bounded.
\end{lemma}
\begin{proof}
  If not, then a disk bubble or sphere bubble forms. By standard compactness arguments, one concludes that $w_{n}$ converges to a (nodal, fibre product) underlying solution after removing small neighborhoods of the bubbling regions.

  The limiting underlying solution $w_{\infty}$ has virtual dimension:
  \begin{equation*}
    \mathrm{vdim}(w_{\infty})=\mathrm{vdim}(w_{n})-\mu(\mathrm{bubbles})-\#\set{\text{additional nodes/breakings}}.
  \end{equation*}
  Because the minimal Maslov number is at least $2$, we have $\mathrm{vdim}(w_{\infty})\le -1$. If $\mathrm{vdim}(w_{\infty})\le -1$, then $x_{-}=x_{+}$, which we exclude from the lemma. Thus there could not have been bubbling.
\end{proof}

\subsubsection{The exceptional case}
\label{sec:gradient-bound-exceptional}

Establishing compactness when $x_{-}=x_{+}$ is more subtle, but it only occurs in the moduli space with $\mu(w)=2$, $k=1$, and $\dim\mathscr{P}=0$ (and is therefore present in the non-equivariant differential). For this, we appeal to the regularity of $(P,J)$ in the sense of \S\ref{sec:maslov-2-disks}. The component of curves with $\mathrm{vdim}(w)=1$ joining $x$ to $x$ is:
\begin{equation*}
  \mathrm{ev}_{2}^{-1}(S_{x}\times U_{x})\setminus \set{z_{1}=z_{2}},
\end{equation*}
where $\mathrm{ev}_{2}$ is the evaluation map of \S\ref{sec:maslov-2-disks}. Thus we work instead with the compactification:
\begin{equation*}
  \Pi_{2}^{x}=\mathrm{ev}^{-1}(S_{x}\times U_{x}),
\end{equation*}
which is the 1-manifold studied in \S\ref{sec:maslov-2-disks}. In this manner, we see that the bubbling off of disks with $\mu=2$ in this exceptional component behaves as a sort of removable singularity; see also \cite{biran-cornea-GT-2012}.

Let us note that $\bd \Pi_{2}^{x}$ can be described explicitly and consists of fibre product curves; see \S\ref{sec:maslov-2-disks}. The count of these boundary points yields the relation:
\begin{equation*}
  \ip{x,d_{eq}^{2}(x)}=0.
\end{equation*}
The gluing result needed to characterize $\bd \Pi_{2}^{x}$ follows from standard transversality methods: the evaluation map $\mathrm{ev}$ is sufficiently submersive that one can make it transverse to $S_{x}\times U_{x}$, and its codimension $1$ boundary strata, and taking preimage of a submanifold with boundary yields a submanifold with boundary.

For the rest of the compactness argument, we will assume that $x_{+}\ne x_{-}$ in the 1-dimensional components.

\subsubsection{Gromov-Floer compactness}
\label{sec:grom-floer-comp}

We are concerned with the compactness of the moduli space quotiented by the $\R^{k+1}$ action. Because of the gradient bound, and because we work in a convex-at-infinity symplectic manifold, we can apply the Arzelà-Ascoli theorem to conclude Floer-type compactness along sequences $w_{n}$ in the component with $\mathrm{vdim}(w_{n})\le 1$. Because the pearls are defined on strips, each pearl can break into a chain of holomorphic strips (this is one way the nodal solutions of \S\ref{sec:nodal-curv} are formed); see Figure \ref{fig:interface-curve}.

One important subtlety introduced by the Borel construction is that the flow line $[\pi_{n}]\in \mathscr{P}/\R$ may diverge (and break in the Morse theory sense). In this case, $w_{n}$ will converge to a fibre product curve; $[\pi_{n}]$ converges in the usual Morse theoretic sense to a configuration of flow lines, but with the caveat that the flow lines appearing in the limit are all transported by self-similarity map $\tau$ and the involution so as to start at $v_{0,+}$ (this is merely a question of interpretation of the limit). This part of the story is treated carefully in \cite{seidel-smith-GAFA-2010,seidel-GAFA-2015}.

Similarly one can have divergence in $\mathscr{G}_{k}$ which corresponds to nodal degenerations (when two points collide), or fibre product curves (when points drift off to $\pm \infty$). Finally, one can always have breakings of Morse flow lines for $P$ at either end, even if $[(\pi_{n},b_{n})]\in (\mathscr{P}\times \mathscr{G}_{k})/\R$ converges. The analysis needed to ``capture'' all the components of the limit is fairly standard by now; we refer the reader to \cite{biran-cornea-arXiv-2007,seidel-smith-GAFA-2010,seidel-GAFA-2015,cant-chen-kyoto-2023} for further details.

\subsubsection{Gluing}
\label{sec:gluing}

There are three types of ``interfaces'' one needs to consider:
\begin{enumerate}
\item\label{item:gluing-1} node removal by gluing disks,
\item node removal by adding flow segments,
\item fibre product gluing.
\end{enumerate}
Each type of result is essentially a gluing result: one uses regularity of the solution $w$ to build ``pre-glued'' approximate solutions $\bar{w}_{R}$ depending on a parameter $R$. For $R$ sufficiently large, one proves the pre-glued solutions $\bar{w}_{R}$ is close to a unique one-parameter family of solutions $w_{R}$ diverging to $w$.

The overall strategy and techniques of gluing are by now standard; see, e.g., \cite{biran-cornea-arXiv-2007,schmaschke-dissertation-2016,brocic-cant-JFPTA-2024,brocic-cant-arXiv-2025}, and we defer to them for further details.

\subsection{Continuation isomorphisms}
\label{sec:cont-isom}

In this section, we explain the invariance of the equivariant quantum cohomology with respect to the choice of Borel data $(P_{\eta},J_{\eta})$. The main results are the following. First, one defines a category $\mathscr{D}$ whose objects are choices of regular Borel data $(P_{\eta},J_{\eta})$ and whose morphisms are homotopy classes of paths of Borel data with fixed endpoints. Then one shows $\mathscr{D}$ is an indiscrete groupoid (the choice of data is contractible). Finally one shows the assignement $(P_{\eta},J_{\eta})\mapsto \mathrm{QH}^{*}_{eq}(P_{\eta},J_{\eta})$ extends to a functor defined on $\mathscr{D}$, taking values in the category of modules over the algebra $\mathbf{k}[q^{-1},q]]\otimes \mathbf{k}[e]$. It follows that all of the quantum cohomologies for different choics of regular data are canonically isomorphic.

\subsubsection{Continuation data}
\label{sec:continuation-data}

We define \emph{Borel continuation data} to be a one-parameter family $P_{\eta,\sigma},J_{\eta,\sigma}$ of Borel data where $\sigma\in \R$. The data should be $\sigma$-independent outside of a finite interval $(-\sigma_{0},\sigma_{0})$. We allow one relaxation: when $\sigma\in (-\sigma_{0},\sigma_{0})$, $P_{\eta,\sigma}$ does not need to be a Morse pseudogradient when $\eta$ is at a pole, (it can be an arbitrary vector field). However, for $\sigma\not\in (-\sigma_{0},\sigma_{0})$, the Borel continuation data is required to be regular in the sense of Definition \ref{sec:linearized-operators}. Let us denote by $P_{\eta,\pm},J_{\eta,\pm}$ the limiting Borel data as $\sigma\to\pm\infty$, and by $P_{\pm},J_{\pm}$ the restrictions of the Borel data to the pole $v_{0,+}$. 

Homotopy classes of Borel continuation data with fixed endpoints are defined in the obvious way, and such homotopy classes can be concatenated, endowing $\mathscr{D}$ with the structure of a category. It is readily verified that for any two choices of regular Borel data, there is a unique homotopy class of continuation data between them, (since the space of almost complex structures and vector fields on $L$ is contractible); thus $\mathscr{D}$ is an indiscrete groupoid.

\subsubsection{Moduli space of continuation trajectories}
\label{sec:moduli-space-cont}

Consider the moduli space of solutions $w=(b,\pi,u_{1},\dots,u_{k})$ to the problem:
\begin{equation*}
  \left\{
    \begin{aligned}
      &b=(b_{1},\dots,b_{k})\in \mathscr{G}_{k}\text{ and }\pi\in \mathscr{P}^{j,\pm},\\
      &\bd_{s}u_{\nu}+J_{\pi(b_{\nu}),b_{\nu}}(u_{\nu})\bd_{t}u_{\nu}=0,\\
      &0<\omega(u_{\nu})<\infty.
    \end{aligned}
  \right.
\end{equation*}
subject to the incidence condition that the flow line of $P_{\pi(s),s}$ for $s\in [b_{\nu},b_{\nu+1}]$ joins $u_{\nu}(+\infty)$ to $u_{\nu+1}(-\infty)$; the asymptotics $x_{-}(w)$ and $x_{+}(w)$ are as in the definition of the equivariant differential. In the case $\pi^{j,-}$, the $x_{+}(w)$ asymptotic is the twist by $a$ of the underlying geometric asymptotic zero; this is necessary so that $x_{\pm}(w)$ are zeros of $P_{\pm}$.

We denote by $\mathscr{M}_{\mathrm{cont}}(x,y)$ the moduli space of such continuation trajectories with $x_{-}(w)=x$ and $x_{+}(w)=y$. Similarly to \S\ref{sec:linearized-operators}, each $w$ determines a certain linearized operator $D_{w}$ defined on:
\begin{equation*}
  W^{1,p,\delta}(u_{1}^{*}TW,u_{1}^{*}TL)\oplus \cdots \oplus W^{1,p,\delta}(u_{k}^{*}TW,u_{k}^{*}TL)\oplus \mathscr{I}_{w},
\end{equation*}
where $\mathscr{I}_{w}$ is a certain subspace of $TL^{2k}\oplus T\mathscr{P}\oplus T\mathscr{G}_{k}$, and $D_{w}$ is a Fredholm operator of index:
\begin{equation*}
  \mathrm{index}(D_{w})=\mu(u_{1})+\dots+\mu(u_{k})+\dim \mathscr{P}+\mathrm{ind}(x_{+})-\mathrm{ind}(x_{-})+k.
\end{equation*}
We say that $w$ is \emph{regular} provided that $D_{w}$ is surjective. Due to the free action by $\R^{k}$, it is convenient to the introduce the virtual dimension as $\mathrm{vdim}(w)=\mathrm{index}(D_{w})-k$.

One defines nodal continuation trajectories in the obvious way, combining the definition in \S\ref{sec:nodal-curv} with the above definition of continuation trajectories. Nodal solutions also have a virtual dimension as in \S\ref{sec:nodal-curv}, the general formula for the virtual dimension is:
\begin{equation*}
  \mathrm{vdim}(w)=\mu(w)+\dim \mathscr{P}+\mathrm{ind}(x_{+})-\mathrm{ind}(x_{-})-\mathfrak{n},
\end{equation*}
where $\mathfrak{n}$ is the number of nodes. Emulating the definition of regular and very regular used in \S\ref{sec:defin-diff}, we have:
\begin{definition}
  Continuation data $(P_{\eta,\sigma},J_{\eta,\sigma})$ is regular provided all (nodal) continuation trajectories with $\mathrm{vdim}(w)\le -\mathfrak{n}$ are regular, and $P_{\eta,\pm},J_{\eta,\pm}$ are regular in the sense of Definition \ref{definition:regular-b-data}. Similarly, continuation data is called very regular provided it is regular and additionally all (nodal) continuation trajectories with $\mathrm{vdim}(w)\le 1-\mathfrak{n}$ are regular.
\end{definition}
In \S\ref{sec:very-regular-continuation-deformation-data-exists} we explain why very regular continuation data is generic, in a manner analogous to Theorem \ref{theorem:regular-b-data}.

\subsubsection{Continuation map}
\label{sec:continuation-map}

The continuation map associated to regular continuation data is defined by appropriately counting the non-nodal continuation trajectories with $\mathrm{vdim}(w)=0$. To be precise, we define:
\begin{equation*}
  \mathfrak{c}(x_{-}):=\sum q^{\mu(w)}e^{\dim\mathscr{P}}x_{+}(w);
\end{equation*}
the sum is over $\R^{k}$-orbits of continuation solutions $w=(b;\pi;u_{1},\dots,u_{k})$ with $\mathrm{vdim}(w)=0$, $\mathfrak{n}=0$, and $x_{-}(w)=x_{-}$. Here $\mu(w)$ is the total Maslov number of the pearls, and $\mathscr{P}$ is the space of flow lines containing $\pi$.

Modulo issues of convergence of sums, this defines a map:
\begin{equation}\label{eq:cmap}
  \mathfrak{c}:\mathrm{CM}(P_{-})\to \mathrm{CM}(P_{+})\otimes \mathbf{k}[q]\otimes \mathbf{k}[e],
\end{equation}
which is extended to a map of $\mathbf{k}[q^{-1},q]]\otimes \mathbf{k}[e]$-modules.

\begin{lemma}
  The map written as \eqref{eq:cmap} is well-defined (the sum converges).
\end{lemma}
\begin{proof}
  Because the sum is only over solutions $w$ with $\mathrm{vdim}(w)=0$, there is a uniform a priori upper bound on $k$ (the number of pearls), $\dim \mathscr{P}$, and $\mu(w)$. Here it is important that the pearls appearing in $w$ are holomorphic (for some $\omega$-tame almost complex structure) so that the Maslov number of each pearl is strictly positive. Thus, in order to prove convergence, it suffices to prove the number of $\R^{k}$-orbits of solutions with fixed $k$, $\mathscr{P}$, and $\mathrm{vdim}(w)=0$ is finite. This is then a standard Gromov-Floer compactness argument; if there was a diverging sequence, one would necessarily conclude a continuation trajectory with negative virtual dimension, contradicting regularity.  
\end{proof}

\subsubsection{Chain map property}
\label{sec:chain-map-property}

The fact that the continuation map is a chain map with respect to $d_{eq}$ is typically proved by analyzing the moduli space of continuation solutions with $\mathrm{vdim}(w)=1$ and $\mathfrak{n}=0$. Similarly to the proof of $d_{eq}^{2}=0$, one considers the nodal solutions with $\mathrm{vdim}(w)=0$ and $\mathfrak{n}=1$ as an interface between different components.

The standard arguments (in the spirit of the proof that $d_{eq}^{2}=0$) show that, for very regular continuation data, it holds that $d_{eq}\circ \mathfrak{c}=\mathfrak{c}\circ d_{eq}$. then:
\begin{theorem}
  For regular continuation data, it holds that $d_{eq}\circ \mathfrak{c}=\mathfrak{c}\circ d_{eq}$.
\end{theorem}
\begin{proof}
  First one proves it for very regular continuation data as explained above. Then one follows the same approximation argument used in \S\ref{sec:proof-theorem-diff-structural} to deduce it for regular data.
\end{proof}

\subsubsection{Homotopies of continuation data give chain homotopic maps}
\label{sec:homot-cont-data}

The argument follows more or less verbatim the argument of \cite{hofer-salamon-95} (in Hamiltonian Floer theory). See also \cite[\S5.1.2]{biran-cornea-arXiv-2007} for a related argument in the setting of pearl complexes.

Let $(P_{\eta,\sigma,\tau},J_{\eta,\sigma,\tau})$ be a homotopy of continuation data, defined for $\tau\in [0,1]$, with fixed endpoints:
\begin{itemize}
\item $P_{\eta,\pm\sigma,\tau}=P_{\eta,\pm,0}$ and $J_{\eta,\pm\sigma,\tau}=J_{\eta,\pm,\tau}$ for $\sigma\ge \sigma_{0}$.
\end{itemize}
Then one can consider the parametric moduli space of pairs $(\tau,w)$ where $w$ solves the (nodal) continuation equation for the data at parameter $\tau$. After quotienting by the $\R^{k}$ action on by translation of the pearls, such solutions have virtual dimension:
\begin{equation*}
  \mathrm{vdim}(\tau,w)=\mu(w)+\dim \mathscr{P}+\mathrm{ind}(x_{+})-\mathrm{ind}(x_{-})+1-\mathfrak{n};
\end{equation*}
provided the solution space is regular (in the parametric sense) the dimension of the quotiented moduli space near $(\tau,w)$ agrees with the above virtual dimension. If the endpoints of the path are regular continuation data, then one can ensure the parametric moduli space is regular by a generic perturbation of $(P_{\eta,\sigma,\tau},J_{\eta,\sigma,\tau})$ for $\tau\in (0,1)$.

The component with $\mathrm{vdim}(w,\tau)=1$ and $\mathfrak{n}=0$, is used to prove the continuation maps defined using the data at endpoints $\tau=0,1$ are chain homotopic (as before, one needs to compactify this component by gluing along the solutions with $\mathrm{vdim}(w,\tau)=0$ and $\mathfrak{n}=1$ considered as interfaces between moduli spaces with different numbers of pearls). The chain homotopy term is defined by counting the component with $\mathrm{vdim}(w,\tau)=0$ and $\mathfrak{n}=0$.

{\itshape Consequently $\mathfrak{c}:\mathrm{QH}^{*}_{eq}(L;P_{\eta,-},J_{\eta,-})\to \mathrm{QH}^{*}_{eq}(L;P_{\eta,+},J_{\eta,+})$ is independent of the choice of regular continuation $(P_{\eta,\sigma},J_{\eta,\sigma})$.}

\subsubsection{Concatenations yield compositions of continuation maps}
\label{sec:concatenations}

This relies on \S\ref{sec:homot-cont-data} and a gluing result similar to those invoked already in \S\ref{sec:gluing}. The outline of the argument is still the same as the one in \cite{hofer-salamon-95}, and we omit the verification that the standard argument works as expected. The outcome is: {\itshape the composition of two continuation maps is a continuation map associated to the concatenation}.

\subsubsection{The continuation endomorphism equals the identity map}
\label{sec:continuation-endomorphism}

Suppose that $(P_{\eta,\sigma},J_{\eta,\sigma})=(P_{\eta},J_{\eta})$ is regular Borel data (for defining $d_{eq}$), considered as $\sigma$-independent continuation data. Then the only solutions which contribute to the continuation map are the constant solutions localized to zeros of $P$, in the component $k=0$ and $\dim \mathscr{P}=0$. It follows that this continuation data is regular for defining $\mathfrak{c}$, and $\mathfrak{c}$ is the identity map. With this established, the functoriality of $(P,J)\mapsto \mathrm{QH}^{*}(L;P,J)$ is proved, and the invariance of $\mathrm{QH}^{*}(L;P,J)$ follows. See \cite[\S5.1.2]{biran-cornea-arXiv-2007} and \cite[\S10.3]{schmaschke-dissertation-2016} for related discussion of invariance.

\subsection{Equivariant displacements}
\label{sec:proof-theorem-displacement}

In this section we prove that $\mathrm{QH}^{*}_{eq}(L)=0$ if $L$ is equivariantly displaceable (Theorem \ref{theorem:displacement}). The idea is to deform the continuation trajectories from \S\ref{sec:cont-isom} using a Hamiltonian perturbation term.

The deformation trajectories we consider will involve holomorphic pearls, except that one pearl is allowed to be non-holomorphic, and is instead a solution of Floer's continuation equation.

\subsubsection{Deformation pearls}
\label{sec:deformation-pearls}

Let $J$ be an $\omega$-tame almost complex structure, and let $X_{s,t}$ be a family of Hamiltonian vector fields whose generator $H_{s,t}$ satisfies $H_{s,t}=0$ for $\abs{s}\ge s_{0}$.

A \emph{deformation pearl} for this data is simply a map:
\begin{equation*}
  \left\{
    \begin{aligned}
      &u:\R\times [0,1]\to (W,L),\\
      &\bd_{s}u+J(u)(\bd_{t}u-X_{s,t}(u))=0,\\
      &\textstyle\int \omega(\bd_{s}u,\bd_{t}u-X_{s,t}(u))<\infty.
    \end{aligned}\right.
\end{equation*}
We note that deformation pearls have well-defined asymptotics $u(\pm\infty)\in L$. Unlike the pearls which appear in the quantum differential or continuation isomorphism, deformation pearls can have negative Maslov number. However, the non-negative energy $E(u)$ (the above integral) is equal to:
\begin{equation*}
  E(u)=\omega(u)+\int \bd_{s}H_{s,t}(u)d s d t,
\end{equation*}
which implies $\omega(u)$ is bounded from below by minus the curvature:
\begin{equation*}
  \mathrm{curvature}=\sup\set{\int \bd_{s}H_{s,t}(v)dsdt},
\end{equation*}
where the supremum is over all smooth maps $v$. The relevant Floer theory constructions we will appeal to naturally yield Hamiltonians $H_{s,t}$ with a priori bounds on their curvature. This curvature bound and monotonicity imply that deformation pearls cannot have a too negative Maslov number; we return to this point in \S\ref{sec:deform-cont-map}.

\subsubsection{Deformation trajectories}
\label{sec:deformation-trajectories}

Define \emph{deformation data} as:
\begin{itemize}
\item a family of Hamiltonians $H_{\eta,s,t}$ as in \S\ref{sec:deformation-pearls} parametrized by $\eta\in S^{\infty}$,
\item Borel data $(P_{\eta},J_{\eta})$ for defining $d_{eq}$ as in \S\ref{sec:borel-data},
\end{itemize}
where $H_{\eta,s,t}$ satisfies the properties:
\begin{enumerate}[label=(H\arabic*)]
\item\label{H1} $H_{-\eta,s,t}=a^{*}H_{\eta,s,t}$,
\item $H_{\tau(\eta),s,t}=H_{\eta,s,t}$,
\item $H_{\eta,s,t}=0$ for $\abs{s}\ge s_{0}$,
\item\label{H4} the curvature of $H_{\eta,s,t}$ is bounded from above by a constant $C_{0}$.
\end{enumerate}
Different choices of deformation Borel data are allowed different choices for the constants $C_{0}$ and $s_{0}$, but these constants should be uniform in the variable $\eta\in S^{\infty}$.

Given such data, consider the solutions $(b,\pi,u_{1},\dots,u_{k})$ to the \emph{deformation equation}:
\begin{equation*}
  \left\{
    \begin{aligned}
      &b\in \mathscr{G}_{k}^{\times}\text{ and }\pi\in \mathscr{P},\\
      &u_{\nu}:\R\times [0,1]\to (W,L),\\
      &\bd_{s}u_{\nu}+J_{\pi(b_{\nu})}(u_{\nu})\bd_{t}u_{\nu}=0\text{ except for }\nu=m\\
      &\mu(u_{\nu})>0\text{ except for }\nu=m.
    \end{aligned}
  \right.
\end{equation*}
The missing pearl $u_{m}$ instead solves the deformation pearl equation:
\begin{equation*}
  \begin{aligned}
    &\bd_{s}u_{m}+J_{\pi(b_{m})}(u_{m})(\bd_{t}u_{m}-X_{\pi(b_{m}),s,t}(u_{m}))=0,
  \end{aligned}
\end{equation*}
and is allowed to have non-positive Maslov number. Each deformation trajectory has exactly one deformation pearl. The deformation pearl may be a constant map. Solutions must satisfy the following incidence condition:
\begin{itemize}
\item the flow line of $P_{\pi(s)}$ on the interval $s\in [b_{i},b_{i+1}]$ takes $u_{i}(+\infty)$ to $u_{i+1}(-\infty)$ for $i=1,\dots,k-1$.
\end{itemize}
The asymptotics $x_{\pm}$ of such a solution are obtained by flowing the initial and terminal points and twisting by $a$ depending on whether $\pi\in \mathscr{P}^{-}$ or $\pi\in \mathscr{P}^{+}$, as in the pearl differential.

As in \S\ref{sec:linearized-operators}, this equation determines a linearized operator acting on:
\begin{equation*}
  W^{1,p,\delta}(u_{1}^{*}TW)\oplus \dots\oplus W^{1,p,\delta}(u_{k}^{*}TW)\oplus \mathscr{I}_{b,\pi,p}
\end{equation*}
where $\mathscr{I}_{b,\pi,p}\subset T\mathscr{G}_{k,b}\oplus T\mathscr{P}_{\pi}\oplus TL^{\oplus 2k}_{p}$ forms the coincident variations: here $p=(u_{1}(-\infty),u_{1}(+\infty),\dots,u_{k}(-\infty),u_{k}(+\infty))$.

By a similar argument to the one given in \S\ref{sec:linearized-operators}, this linearized operator is a Fredholm operator with index:
\begin{equation*}
  \mathrm{index}(w)=\mu(w)+\mathrm{ind}(x_{+})-\mathrm{ind}(x_{-})+k+\dim \mathscr{P}.
\end{equation*}
Now there is only an $\R^{k}$-action on the moduli space rather than a $\R^{k+1}$-action --- this is because the deformation pearl should not be translated.

Similarly to \S\ref{sec:nodal-curv}, one also considers nodal solutions, where nodal chains of holomorphic disks located at the same position $b_{i}$ are allowed (this includes chains of holomorphic pearls at either end of the deformation pearl). The virtual dimension formula for a configuration with $k$ pearls and $\mathfrak{n}$ nodes is:
\begin{equation*}
  \mathrm{vdim}(w)=\mathrm{index}(w)-k=\mu(w)+\mathrm{ind}(x_{+})-\mathrm{ind}(x_{-})+\dim \mathscr{P}-\mathfrak{n}.
\end{equation*}

This leads to the definition of regular deformation data:
\begin{definition}\label{definition:very-regular-deformation-data}
  Data $(P_{\eta},J_{\eta},H_{\eta,s,t})$ is said to be regular deformation data provided that:
  \begin{enumerate}
  \item $P_{\eta},J_{\eta}$ is regular Borel data for defining the quantum differential,
  \item each (nodal) deformation trajectory with $\mathrm{vdim}(w)\le -\mathfrak{n}$ has a surjective linearized operator.
  \end{enumerate}
  If instead we replace the second item by:
  \begin{enumerate}[resume]
  \item each (nodal) deformation trajectory with $\mathrm{vdim}(w)\le 1-\mathfrak{n}$ has a surjective linearized operator,
  \end{enumerate}
  then the data is said to be very regular.
\end{definition}
Let us note that that the constant data $H_{\eta,s,t}=0$ is very regular deformation data provided $(P_{\eta},J_{\eta})$ is regular data for the quantum differential; the argument is the same one given in \S\ref{sec:continuation-endomorphism}. In \S\ref{sec:very-regular-continuation-deformation-data-exists} we claim that very regular deformation data is generic, in a manner analogous to Theorem \ref{theorem:regular-b-data}.

\subsubsection{The deformed continuation map}
\label{sec:deform-cont-map}

Let $(P_{\eta},J_{\eta},H_{\eta,s,t})$ be regular deformation data. Then one can count the moduli space of orbits of the $\R^{k}$ action on the component with $\mathrm{vdim}(w)=0$ and $\mathfrak{n}=0$. In this count, one considers all the moduli spaces, allowing any possible choice for:
\begin{itemize}
\item the number of pearls $k$, 
\item the flow line space $\mathscr{P}$ on $S^{\infty}$ starting at $v_{0,+}$,
\item the total Maslov number, and,
\item the position $\ell\in \set{1,\dots,k}$ of the deformation pearl. 
\end{itemize}
The count takes a generator $x_{-}\in \mathrm{CM}(P)$ as input, and returns an element of $\mathrm{CM}(P)\otimes \mathbf{k}[q^{-1},q]]\otimes \mathbf{k}[e]$ as output:
\begin{equation*}
  \mathfrak{h}(x_{-}):=\sum e^{\dim \mathscr{P}}q^{\mu(w)}x_{+}(w).
\end{equation*}
We recall that, as always, the asymptotics are appropriately twisted by the involution $a$, as explained in \S\ref{sec:defin-diff}.
\begin{lemma}\label{lemma:finiteness}
The sum over the $\mathrm{vdim}(w)=\mathfrak{n}=0$ component is finite for regular deformation data.
\end{lemma}
\begin{proof}
  The crux of the matter is bounding $\dim \mathscr{P}$ and $\mu$ from above. Once these bounds are established, it is easy to prove compactness following a standard argument we outline at the end of the proof.

  The vanishing of the virtual dimension, and the positivity of $\dim \mathscr{P}$ and $\mu(u_{i})$ for holomorphic pearls show that it is sufficient to bound $\mu(u_{m})$ from below for the deformation pearl.

  We have $\mu(u_{m})=N\omega(u_{m})$ after rescaling $\omega$ if necessary. The energy identity from Floer theory \S\ref{sec:deformation-pearls} gives:
  \begin{equation*}
    \omega(u_{m})>-(\text{curvature of $H_{\pi(b_{m}),s,t}$})\ge -C_{0}
  \end{equation*}
  where $C_{0}$ is from the hypotheses of deformation data \S\ref{sec:deformation-trajectories}. It therefore follows that $\mu(u_{m})$ is bounded from below by $-NC_{0}$. Then we have:
  \begin{equation*}
    \dim \mathscr{P}\le \dim L+NC_{0}.
  \end{equation*}
  Similarly the total Maslov number of the holomorphic pearls is bounded from above by $\dim L+NC_{0}$. This also implies that the number of pearls $k$ is bounded from above.

  Now suppose one has a failure of compactness represented as a sequence $w_{n}$. One can pass to a subsequence so that $\mathscr{P}$, $k$, and the choice $m\in \set{1,\dots,k}$ are fixed along $w_{n}$.
  
  Then, this failure of compactness must produce:
  \begin{itemize}
  \item a limiting (nodal) solution with $\mathrm{vdim}(w_{\infty})<0$ which is regular;
  \end{itemize}
  note that the limiting solution may be non-nodal if it appears as one component of a Morse-type breaking (fibre product curve in \S\ref{sec:fibre-prod-curves}) or by some bubbling off of holomorphic spheres or disks. Such a limiting solution cannot exist by regularity. If a nodal solution develops it is straightforward to verify that $\mathrm{vdim}(w_{\infty})\le -\mathfrak{n}$, and thus the limiting solution is automatically regular.
\end{proof}

\subsubsection{Chain map property}
\label{sec:chain-map-property-1}

It follows from the preceding Lemma that $\mathfrak{h}$ is a well-defined endomorphism of $\mathrm{CM}(P)\otimes \mathbf{k}[q^{-1},q]]\otimes \mathbf{k}[e]$ when it defined using regular data. We show:
\begin{lemma}\label{lemma:chain-map}
  It holds that $\mathfrak{h}\circ d_{eq}=d_{eq}\circ \mathfrak{h}$ for regular data.
\end{lemma}
\begin{proof}
  As explained in \S\ref{sec:proof-theorem-diff-structural}, it suffices to prove it for very regular data (see Definition \ref{definition:very-regular-deformation-data}). We count the ends of the moduli space of deformation trajectories in the component $\mathrm{vdim}(w)=1$ and $\mathfrak{n}=0$. These ends come in two types:
  \begin{enumerate}
  \item fibre-product configurations contributing to $\mathfrak{h}\circ d_{\mathrm{eq}}$ or $d_{\mathrm{eq}}\circ \mathfrak{h}$, (which case is determined by where the deformation pearl ends up),
  \item nodal configurations with $\mathrm{vdim}(w)=0$ and $\mathfrak{n}=1$.
  \end{enumerate}
  Other failures of compactness can be precluded using the arguments in the proof of Lemma \ref{lemma:finiteness}. By the aforementioned gluing results \S\ref{sec:gluing}, each end of the second type appears exactly twice, and each end of the first type appears exactly once. One concludes the desired relation.
\end{proof}

\subsubsection{Chain homotopy property}
\label{sec:chain-homot-prop}

It follows that $\mathfrak{h}$ induces a well-defined endomorphism of the homology $\mathrm{QH}_{eq}^{*}(L;P_{\eta},J_{\eta})$. We show:

\begin{lemma}
  It holds that $\mathfrak{h}=\id$ as maps on $\mathrm{QH}_{eq}^{*}(L;P_{\eta},J_{\eta})$, provided that $(P_{\eta},J_{\eta},H_{\eta,s,t})$ is regular deformation data.
\end{lemma}
\begin{proof}
  First we claim that the chain homotopy class of $\mathfrak{h}$ depends only on the homotopy class of $H_{\eta,s,t}$ in the space of data satisfying \ref{H1} through \ref{H4}. The argument establishing this is an application of parametric moduli spaces; see \S\ref{sec:homot-cont-data}. Then we claim that $\tau \mapsto (P_{\eta},J_{\eta},\tau H_{\eta,s,t})$ is a homotopy between $(P_{\eta},J_{\eta},0)$ and $(P_{\eta},J_{\eta},H_{\eta,s,t})$, satisfying \ref{H1} through \ref{H4} throughout the homotopy. By the same argument in \S\ref{sec:continuation-endomorphism}, it follows that the map $\mathfrak{h}$ defined using $(P_{\eta},J_{\eta},0)$ is the identity map on chain level; the desired result follows.
\end{proof}

\subsubsection{Vanishing property}
\label{sec:vanishing-property}

We now come to the crucial property, that $\mathfrak{h}$ vanishes in the presence of an equivariant displacement of $L$. This section will complete the proof of Theorem \ref{theorem:displacement}. In fact, we will prove a seemingly stronger statement. Suppose that $H_{\eta,t}$ is a family satisfying:
\begin{itemize}
\item $H_{-\eta,t}=a^{*}H_{\eta,t}$,
\item $H_{\tau(\eta),t}=H_{\eta,t}$,
\item $\textstyle \int \max H_{\eta,t}-\min H_{\eta,t} dt$ is uniformly bounded from above.
\end{itemize}
Then, if we define, for a standard cut-off function $\beta$,
\begin{equation}\label{eq:cut-off-formula}
  H_{\eta,s,t}=\beta(s_{1}-s)\beta(s+s_{1})H_{\eta,t},
\end{equation}
we obtain a system satisfying \ref{H1} through \ref{H4}, for any choice of $s_{1}$. Moreover, the curvature bound is independent of $s_{1}$. One obvious choice of $H_{\eta,t}$ is $H_{\eta,t}=H_{t}$ where $H_{t}$ is equivariant. To ensure regularity, one allows a small $\eta$-dependent perturbation which can be made without increasing the curvature by more than any small amount $\epsilon$; these perturbations should be supported outside $[1-s_{1},s_{1}-1]$. Then:

\begin{lemma}
  If the isotopy generated by $H_{\eta,t}$ displaces $L$ for each $\eta$, then the deformation map $\mathfrak{h}$ defined using \eqref{eq:cut-off-formula} vanishes for $s_{1}$ sufficiently large.
\end{lemma}
\begin{proof}
  The energy of the deformation pearl $u_{m}$ restricted to $s\in [1-s_{1},s_{1}-1]$ is bounded from above by $\omega(u_{m})$ plus the curvature of $H_{\pi(b_{m}),s,t}$. Both quantities are uniformly bounded (the bound on $\omega(u_{m})$ comes from the bound on $\mu(u_{m})$). On this interval, $u_{m}$ solves Floer's equation for $H_{\pi(b_{m}),t}$.

  On the other hand, it follows that strips of length $1$ solving Floer's equation for $H_{\pi(b_{m}),t}$ cannot have an arbitrarily small energy (since this isotopy displaces $L$) --- one has compactness for the possible values of $\eta=\pi(b_{m})$ because of the a priori upper bound on $\dim \mathscr{P}$. Therefore, for $s_{1}$ large enough, the length of $[1-s_{1},s_{1}-1]$ is too long for any deformation pearl to exist without having too much energy, and so the map $\mathfrak{h}$ must vanish identically on chain level. This completes the proof of Theorem \ref{theorem:displacement}.
\end{proof}

\section{The Biran-Khanevsky construction and the proof of Theorem \ref{theorem:main-computation}}
\label{sec:proof-theorem-main-computation}

To prove Theorem \ref{theorem:main-computation}, we introduce a special class of Borel data $(P_{\eta},J_{\eta})$ which we call \emph{Biran-Khanevsky type}, inspired by the paper \cite{biran-khanevsky-CMH-2013}. To construct such data, we fix once and for all data $(Q,J)$ for defining the quantum cohomology complex of $\bar{L}$ inside of the reduction $M=\bd\Omega/(\R/\Z)$. As $L\to \bar{L}$ is a flat circle bundle, there exists a unique flat lift of $Q$ to $TL$, which we denote by $P_{\mathrm{flat}}$.

Unfortunately $P_{\mathrm{flat}}$ is not a Morse pseudogradient vector field. We explain how to modify this. Pick disjoint coordinate disks around each zero $z$ of $Q$, on which the flow is diagonal and linear, with radial coordinate $\rho$, and trivialize the bundle over these neighborhoods, obtaining a new fibre coordinate $x_{0}\in \R/\Z$. Define:
\begin{equation}\label{eq:formula-for-fibre}
  V_{z,\epsilon,\theta}=\beta(1-\epsilon^{-1}\rho)\sin(2\pi (x_{0}-\theta))\bd_{x_{0}},
\end{equation}
where $\beta$ is a standard cut-off function. Given $\theta:\set{\text{zeros of $Q$}}\to \R$, define the pseudogradient:
\begin{equation*}
  P_{\epsilon,\theta}=P_{\mathrm{flat}}+\sum_{z}V_{z,\epsilon,\theta(\zeta)},
\end{equation*}
This vector field is a Morse pseudogradient, is related to $Q$ under the projection $\bd\Omega\to M$, and there are two critical points $x(z),y(z)$ in the fibre over each zero $z$ of $Q$, where $\mathrm{ind}(y(z))=\mathrm{ind}(x(z))+1$.

For the remainder of this section, we will consider $\epsilon\in (0,1)$ fixed and $\theta$ as a generic choice, and abbreviate $P=P_{\epsilon,\theta}$, so that $a^{*}P=P_{\epsilon,\theta+0.5}$. The following lemma is proved in \S\ref{sec:morse-smale-property}.

\begin{lemma}\label{lemma:morse-smale}
  Fix $\epsilon\in (0,1)$. There is a dense open set of choices of $\theta$ which render $P_{\epsilon,\theta}$ a Morse-Smale pseudogradient.
\end{lemma}

Then the genericity results in \S\ref{sec:borel-data} imply $P$ can be extended to some regular Borel data $(P_{\eta},J_{\eta})$. However, we will need to pick our extension $(P_{\eta},J_{\eta})$ from a highly restricted class. To state the condition, introduce the symplectization end $SY\subset W$ with radial coordinate $r$ so that $r=1$ is $\bd\Omega$. The circle action leads to a special SFT type almost complex structure:
\begin{definition}
  Given $J$ on the reduction $Q$, define $J_{S}$ on $SY$ to be the unique almost complex structure satisfying:
  \begin{itemize}
  \item $J_{S}$ satisfies $J_{S}Z=R$ where $Z$ is the Liouville vector field,
  \item the projection map $SY\to Q$ is holomorphic from $J_{S}$ to $J$;
  \end{itemize}
  This latter projection map is the quotient by the Liouville and Reeb flows.
\end{definition}

\begin{definition}\label{definition:bk-type}
  A Borel data $(P_{\eta},J_{\eta})$ satisfying $P_{v_{0,+}}=P$ is said to be of Biran-Khanevsky type provided that:
  \begin{itemize}
  \item $J_{\eta}=J_{S}$ (SFT type) on the region $r\ge 1$,
  \item $P$ is related to $Q$ under the projection $L\to \bar{L}$.
  \end{itemize}
\end{definition}
Working in a Liouville manifold opens the possibility of neck stretching: let us say that $J_{\eta}'$ is obtained from $J_{\eta}$ by \emph{neck-stretching} with parameter $s$ provided that:
\begin{equation}\label{eq:neck-stretching}
  J_{\eta}'(z)=(\d\rho_{s,z})^{-1}J_{\eta}(\rho_{s}(z))\d\rho_{s,z}\text{ where }\rho_{s}\text{ is the time $s$ Liouville flow.}
\end{equation}
There are two key lemmas concerning such data:
\begin{lemma}\label{lemma:neck-stretching}
  Let $(P_{\eta},J_{\eta})$ be of Biran-Khanevsky type. After sufficiently large neck-stretching of $J_{\eta}$, all solutions $w$ of the equivariant differential equation \S\ref{sec:defin-diff} with $\mathrm{vdim}(w)=0$, see \eqref{eq:virtual-dimension-differential}, have the property that all pearls $u_{i}$ remain a minimum positive distance away from the skeleton,\footnote{The skeleton is the complement of the symplectization end.} and lie in the region where $J_{\eta}=J_{S}$, and therefore can be projected down to $J$-pearls on $\bar{L}$.
\end{lemma}
\begin{proof}
  See \cite[Proposition 5.1]{biran-khanevsky-CMH-2013}. We recall the argument in \S\ref{sec:proof-of-neck-stretching}.
\end{proof}

Henceforth we assume our data is sufficiently neck stretched so that Lemma \ref{lemma:neck-stretching} applies. The next two important results we will require from \cite{biran-khanevsky-CMH-2013} are proved in \S\ref{sec:proof-of-unique-lifting} and \S\ref{sec:lift-pearl-traj}; see \cite[\S 7]{biran-khanevsky-CMH-2013}.

\begin{lemma}[Unique lifting of disks]\label{lemma:unique-lifting}
  Let $J_{S}$ be the SFT type almost complex structure on $SY$, and let:
  \begin{equation*}
    v:(D,\bd D)\to (M,\bar{L})
  \end{equation*}
  be a $J$-holomorphic disk. Then there is a unique lift $u:(D,\bd D)\to (SY,L)$ up to the Reeb flow (i.e., once the lift of any point $z\in \bd D$ is chosen in the circular fibre above $z$ in the level $r=1$). 
\end{lemma}
\begin{lemma}\label{lemma:BK-is-regular}
  For generic $\theta$, and sufficiently neck stretched $J_{\eta}$, the data $(P_{\eta},J_{\eta})$ is regular in the sense of Definition \ref{definition:regular-b-data}.
\end{lemma}

Turning now to the equivariant differential, let us set up some notation: if $P=P_{\epsilon,\theta}$, then the chain complex is:
\begin{equation*}
  \mathrm{CM}(P)=\mathrm{CM}(Q)\oplus \mathrm{CM}(Q)[-1];
\end{equation*}
and above each generator $z\in \mathrm{CM}(Q)$, recall there are two generators:
\begin{itemize}
\item $y(z)$ with $\mathrm{ind}(y(z))=\mathrm{ind}(z)+1$,
\item $x(z)$ with $\mathrm{ind}(x(z))=\mathrm{ind}(z)$.
\end{itemize}
This induces a splitting of the total complex:
\begin{equation*}
  \mathrm{CM}(P)\otimes \mathbf{k}[x^{-1},x]]\otimes \mathbf{k}[e]\simeq (\mathrm{CM}(Q)\oplus \mathrm{CM}(Q))\otimes \mathbf{k}[q^{-1},q]]\otimes \mathbf{k}[e]
\end{equation*}

Our main computational theorem in this section is proved in \S\ref{sec:comp-diff}.
\begin{theorem}\label{theorem:differential}
  For generic $\theta$ and generic extension $P_{\eta}$ the equivariant differential computed using sufficiently neck stretched $(P_{\eta},J_{\eta})$ takes the form:
  \begin{equation}
    d_{eq}=\begin{bmatrix} d & e^2+F\\
      0 & d
    \end{bmatrix}
  \end{equation}
  where $d$ is the quantum differential on $\mathrm{CM}(Q)\otimes \mathbf{k}[q^{-1},q]]$, and $F$ is an endomorphism of $\mathrm{CM}(Q)\otimes \mathbf{k}[q^{-1},q]]$  that agrees with multiplication with the Floer-Euler class on quantum cohomology (see \S\ref{sec:biran-khanevnsky-floer-gysin-sequence}). Both $d$ and $F$ are extended by linearity over $\mathbf{k}[e]$.
\end{theorem}

In the final subsection \S\ref{sec:proof-theorem-main} we use Theorem \ref{theorem:differential} to prove Theorem \ref{theorem:main-computation}.

\subsection{The Morse-Smale property}
\label{sec:morse-smale-property}

We prove Lemma \ref{lemma:morse-smale}. Each flow line $\bar{\gamma}$ of the underlying pseudogradient $Q$ lifts to a $\R/\Z$-family of flow lines $\gamma$ for the pseudogradient $P_{\epsilon,\theta}$; one simply lifts the initial point $\bar{\gamma}(0)$. Suppose first that $z_{-}\ne z_{+}$. The space of lifts lying over $\mathscr{P}(z_{-},z_{+})$ is a circle bundle $\mathscr{L}(z_{-},z_{+})\to \mathscr{P}(z_{-},z_{+})$, defined by pulling back $L\to \bar{L}$ via $\bar{\gamma}\mapsto \bar{\gamma}(s_{0})$ where $s_{0}$ is the first time when $\bar{\gamma}(s_{0})$ exits the disk $D(z_{-})$.

The local model \eqref{eq:formula-for-fibre} implies:
\begin{itemize}
\item there is a unique lift $\gamma_{y}$ of $\bar{\gamma}$ starting at $y(z_{-})$,
\item there is a unique lift $\gamma_{x}$ of $\bar{\gamma}$ ending at $x(z_{+})$.
\end{itemize}
These can be interpreted as smooth sections of $\mathscr{L}(z_{-},z_{+})\to \mathscr{P}(z_{-},z_{+})$. In particular, we can compare the two lifts and obtain an $\R/\Z$-valued function $t_{0}$; prosaically, $\gamma_{y}(s)=R_{t_{0}}\gamma_{x}(s)$ for some $t_{0}\in \R/\Z$, where $R$ acts by the rotation in the fiber. This produces a smooth map:
\begin{equation*}
  f_{z_{-},z_{+}}:\set{\bar{\gamma}\text{ joining }z_{-}\text{ to }z_{+}}\to \R/\Z\text{ given by }f_{z_{-},z_{+}}(\bar{\gamma})=t_{0}.
\end{equation*}
\textbf{Claim A.} {\itshape For a dense and open set of $\theta$, it holds that $f_{z_{-},z_{+}}$ has zero as a regular value for each pair $z_{-}\ne z_{+}$.}

This follows immediately from Sard's theorem \cite{milnor-book-topology-diff-viewpoint-1965}, as rotating $\theta$ at $z_{-}$ changes $f_{z_{-},z_{+}}$ by addition of a constant (here we assume that the disks $D_{z_{-}}$ is sufficiently small so that any flow line starting at $z_{-}$ never re-enters $D_{z_{-}}$ after it leaves). Next:

\textbf{Claim B.} {\itshape If $f_{z_{-},z_{+}}$ has zero as a regular value for each pair $z_{-}\ne z_{+}$, then $P_{\epsilon,\theta}$ is Morse-Smale.}

The technical definition of Morse-Smale requires consideration of the linearized operators, which we do at the end of the proof. First we show that the moduli spaces are manifolds of the correct dimension.

Inside of $\mathscr{P}(z_{-},z_{+})$ there is the submanifold $\mathscr{R}(z_{-},z_{+})\subset \mathscr{P}(z_{-},z_{+})$ cut out by $f_{z_{-},z_{+}}^{-1}(0)$. In fact, we have constructed two sections $\gamma_{x}$ and $\gamma_{y}$ of the circle bundle $\mathscr{L}(z_{-},z_{+})$, and $\mathscr{R}(z_{-},z_{+})$ is identified with their transverse intersection $\gamma_{x}\cap \gamma_{y}$. By inspection, there are isomorphisms:
\begin{equation}\label{eq:isomorphisms-describing-flow-spaces}
  \begin{aligned}
    &\gamma_{x}:\mathscr{P}(z_{-},z_{+})- \mathscr{R}(z_{-},z_{+})\to \mathscr{P}(x(z_{-}),x(z_{+}))\\
    &\gamma_{y}:\mathscr{P}(z_{-},z_{+})- \mathscr{R}(z_{-},z_{+})\to \mathscr{P}(y(z_{-}),y(z_{+}))\\
    &\gamma_{x}=\gamma_{y}:\mathscr{R}(z_{-},z_{+})\to \mathscr{P}(y(z_{-}),x(z_{+}))\\ 
    &\id:(\mathscr{L}(z_{-},z_{+})-(\gamma_{x}\cup \gamma_{y}))\to \mathscr{P}(x(z_{-}),y(z_{+}));\\
  \end{aligned}
\end{equation}
in the final line, the domain is the total space of a circle bundle with the images of two sections removed. Each space of trajectories is identified with a manifold of the correct dimension. To complete the proof of the claim, we linearize the problem and compute the kernel of the linearized operator.

Let $\gamma$ be a flow line lifting $\bar{\gamma}$. There is a canonical splitting using the horizontal distribution $\xi\cap TL$ as a (flat) Ehresmann connection:
\begin{equation*}
  TW^{1,p}(\R,L)_{\gamma}\simeq W^{1,p}(\R,\R)\oplus TW^{1,p}(\R,\bar{L})_{\bar{\gamma}}.
\end{equation*}

\textbf{Claim C.} {\itshape With respect to this splitting, the linearized operator takes the form:
\begin{equation*}
  D_{\gamma}=\left[
  \begin{matrix}
    D_{\R}&A\\
    0&D_{\bar{\gamma}}\\
  \end{matrix}
  \right]
\end{equation*}
where $D_{\R}$ is a Fredholm operator of index $0,1,-1$ depending on whether the asymptotics are of $x$ type or $y$ type. To be precise, the index is:
\begin{enumerate}[label=(\roman*)]
\item\label{ms-1} $1$ if $\gamma\neq \gamma_{x}$ and $\gamma\neq \gamma_{y}$,
\item\label{ms-2} zero if $\gamma=\gamma_{x}$ or $\gamma=\gamma_{y}$, but not both, and,
\item\label{ms-3} $-1$ if $\gamma=\gamma_{x}=\gamma_{y}$.
\end{enumerate}
}

Let us note that the claim about the Fredholm index of $D_{\R}$ follows from the general properties of Morse theory, once we prove that the linearization $D_{\Gamma}$ splits in the stated fashion. To prove the splitting, we briefly recall how the linearization is defined, following \cite[\S4]{cant-thesis-2022}.

Pick a coordinate chart $x_{1},\dots,x_{n-1}$ on the reduced space $\bar{L}$, and lift this to a leaf of $\xi\cap TL$, which is a flat distribution. Then we define an additional $\R/\Z$-valued coordinate $x_{0}$ by rotating the point on the chosen leaf. By our construction, the vector field $P_{\epsilon,\theta}$ appears in the form:
\begin{equation*}
  V(x_{0},x_{1},\dots,x_{n-1})\pd{}{x_{0}}+\sum Q_{i}(x_{1},\dots,x_{n-1})\pd{}{x_{i}}.
\end{equation*}
Then, if $\eta$ points in the $\partial/\partial{x_{0}}$ direction, it holds that:
\begin{equation*}
  D_{\gamma}^{x}(\eta)=\eta'(s)-\pd{V}{x_{0}}\eta(s),
\end{equation*}
where $D^{x}_{\gamma}$ is the local coordinate linearization.\footnote{The local coordinate linearization is computed in the obvious way:
  \begin{equation*}
    D_{\gamma}^{x}(\eta)=\eta'-\lim_{\epsilon\to 0}\epsilon^{-1}\left(\d x P x^{-1}(x\gamma(s)+\epsilon \eta(s))-\d x P(\gamma(s))\right);
  \end{equation*}
  see, e.g., \cite[\S4.2]{cant-thesis-2022}. This is only defined for $\eta$ compactly supported in the open set $\gamma^{-1}(U)$ where $U$ is the domain of the coordinates $x$. These local linearized operators uniquely determine the global linearized operator $D_{\gamma}$ by the relation $\d x\circ D_{\gamma}\circ \d x^{-1}=D^{x}_{\gamma}$.
} The genuine linearized operator is related to the local linearization by $\d x\circ D_{\gamma}\circ \d x^{-1}=D^{x}_{\gamma}$. Then, since $\d x$ respects the splitting, we conclude the claim.

From this splitting, we see that the reduction map takes elements in the kernel of $D_{\gamma}$ into the kernel of $D_{\bar{\gamma}}$. It is clear from basic ODE theory that $\dim \ker D_{R}\le 1$. Thus in case \ref{ms-1} we conclude: $$\dim\ker D_{\gamma}\le \dim \ker D_{\bar{\gamma}}+\dim \ker D_{\R}\le \dim \ker D_{\bar{\gamma}}+1=\mathrm{index}(D_{\gamma}),$$ proving regularity in this case. In case \ref{ms-2} and \ref{ms-3} we differentiate the inverses of the lifting maps in \eqref{eq:isomorphisms-describing-flow-spaces} and thereby prove:
\begin{equation*}
  \dim \ker D_{\gamma}-\mathrm{index}(D_{\gamma})\le \dim \ker D_{\R}.
\end{equation*}
However, in both cases $\dim \ker D_{\R}=0$, by analysis of linear ODEs on $W^{1,p}(\R,\R)\to L^{p}(\R,\R)$ with non-degenerate asymptotics.\footnote{Here we implicitly appeal to structural results in Morse theory. In particular, we know that $D_{\R}$ is represented by $\bd_{s}+S(s)$ where $S(s)$ converges to non-zero numbers at both ends, and the signs of these non-zero numbers can be determined by computing Morse indices. Such ODE can be solved explicitly; see, e.g., \cite[\S2.4]{cant-thesis-2022}.} Thus we have regularity in the cases $z_{-}\ne z_{+}$.

It remains only to treat the case $z_{-}=z_{+}$. In this case the same arguments (considering an upper triangular linearization) work in the same way to prove $D_{\gamma}$ is transverse. This completes the proof of the lemma. \hfill$\square$

\subsection{Neck stretching}
\label{sec:proof-of-neck-stretching}

This section recalls the proof of Lemma \ref{lemma:neck-stretching}. We first show that solutions appearing in the moduli space with $\mathrm{vdim}(w)=0$ must map all of their pearls into the symplectization end $SY$, into the region where $J_{\eta}=J_{S}$, provided that $J_{\eta}$ is sufficiently neck stretched.

To be precise, let us pick a convergent family of pseudogradients $P_{\eta,n}$, and complex structures $J_{\eta,n}$ which are getting more and more neck-stretched, and suppose (in search of a contradiction) that there is a sequence of solutions $w_{n}$ which have pearls intersecting the $1/n$-neighborhood of the skeleton. One shows $w_{n}$ converges to a pearl trajectory where disks are replaced by pseudoholomorphic SFT type buildings. The argument which derives the contradiction is well-explained in \cite[\S5]{biran-khanevsky-CMH-2013}, so we only highlight the salient points:
\begin{itemize}
\item the ``top-level'' building projects to a genuine solution $\bar{w}$ in the reduction $M$, because being asymptotic to a Reeb orbit in the negative end implies the pearls in $\bar{w}$ are punctured disks with removable singularities,
\item the ``middle components'' of the building are non-constant rational holomorphic curves in $SY$ which project to non-constant holomorphic spheres in $M$,
\item the ``deepest components'' of the building are non-constant rational curves in $W$ using the original (non-neckstretched) almost complex structure $J_{\eta}$.
\end{itemize}
Each deepest component is a map $v:\C\to W$. Without loss of generality, let us suppose $v$ is transverse to $r=1$ (otherwise just perturb slightly $r=1+\epsilon$). Then $v^{-1}(\set{r=1})$ is a collection of disjoint circles in $\C$, which bound disks $D_{1},\dots,D_{k}$. If $D_{i}\subset D_{j}$ then we just remove $D_{i}$ from the list.

We can homotope $v$ on the union of the disks $D=D_{1}\cup \dots\cup D_{k}$ so that it takes values in $SY$, by our assumption that $\pi_{2}(\Omega,\bd\Omega)=0$. Denote this deformation by $v_{1}$. It is crucial to set $\alpha=\lambda/r$ and estimate:
\begin{equation}\label{eq:positivity-integral-alpha}
  \begin{aligned}
    \int v_{1}^{*}\d \alpha&=\int_{D^{c}}v^{*}\d \alpha+\int_{D}v_{1}^{*}\d\alpha\\
    &=\int_{D^{c}}v^{*}\d \alpha+\int_{D}v^{*}\d\lambda>0.
  \end{aligned}
\end{equation}
The inequality uses Stokes' theorem and ideas of \cite{BEHWZ}. The significance of the computation is that, by definition of symplectic reduction, the $\omega_{M}$-area of the projection of $v_{1}$, say $\bar{v}_{1}$, is given by:
\begin{equation*}
  \int \bar{v}_{1}^{*}\omega_{M}=\int v_{1}^{*}\d \alpha.
\end{equation*}
Indeed, $\bar{v}_{1}^{*}\omega_{M}$ is computed by projecting $\tilde{v}$ onto the level $r=1$, and then integrating $\d \lambda$ over this projection of $\tilde{v}$. But $\d\lambda=\d\alpha$ on the level $r=1$. It therefore follows from \eqref{eq:positivity-integral-alpha} that $c_{1}(\bar{v})>0$, since the reduction $\bar{v}$ has positive $\omega_{M}$-area (here we use the monotonicity of $M$ in a seemingly crucial way rather than the weaker assumption that $\mu$ is positive on $J$-holomorphic curves).

In this manner we observe that $\mu(\bar{w})$ must be strictly lower than $\mu(w_{n})$, because the deepest components or middle components always consume a positive amount of intersections with the first Chern class. One then concludes from the fact that $w_{n}$ was in a moduli space with $\mathrm{vdim}(w_{n})=0$ (and the fact the minimal Maslov number is at least $2$) that $\bar{w}$ lives in a moduli space with virtual dimension $-1$ or less, which is impossible since $\bar{w}$ needs to have at least one pearl. \hfill$\square$

\subsection{Unique lifting of holomorphic disks}
\label{sec:proof-of-unique-lifting}

We prove Lemma \ref{lemma:unique-lifting}. As in the setting of the previous section, after sufficient neck stretching we may assume that all the relevant pearls lie in the symplectization end $SY\subset W$ (i.e., we use a sufficiently neck-stretched almost complex, so that all pearls lie in the region where the almost complex structures agrees with $J_{S}$). Set $\pi:SY\to M$ to be the reduction map; here $M=SY/\C^{\times}$ is the quotient by the Liouville and Reeb flows. Then $\pi:SY\to M$ is a principal bundle with structure group $\C^{\times}$, and the projection map is $(J_{S},J)$-holomorphic.

Let $\bar{u}:(D,\partial D)\to (M, \bar{L})$ be a holomorphic disk in $M$. The pullback bundle $\bar{u}^{\ast}SY$ admits a section $v$ (probably non-holomorphic) taking values in $\bar{u}^{*}L$ along $\partial D$. To construct a holomorphic lift $u$ of $\bar{u}$ we consider an ansatz of the form:
\begin{equation}\label{eq:conformal-ansatz}
  u=e^{f}v\text{ where }f:(D,\bd D)\to (\C, \R)\text{ is a smooth function},
\end{equation}
and where $e^{f}\in \C^{\times}$ acts on $v$ as explained above.

\begin{lemma}\label{lemma:conformal-ansatz}
  There is unique solution $f$ which renders $u$ holomorphic, up to the additions of imaginary constants.
\end{lemma}
\begin{proof}
The action of $e^{a+ib}\in \C^{\times}$ on a point $p\in \Omega$ is given by
$$e^{a+ib}\cdot p= \rho_{a}\circ \psi_{b}(p)$$
where $\rho_{a}$ is the time $a$ flow of the Liouville vector field $Z$ and $\psi_{b}$ is the time $b$ flow of the Reeb vector field $R$ (the flows commute so the order is immaterial).

Set $f=a+ib:(D,\partial D)\to (\C, i\R)$ and consider the map $u=e^{f}v$ which coincides with $\rho_{a}\circ \psi_{b}(v)$. Compute
$$\begin{cases}
  \partial_{s}u= d\rho_{a}\circ d\psi_{b}(\partial_s v)+ \partial_s a\cdot Z(u)+\partial_{s} b \cdot R(u)\\
  \partial_{t}u= d\rho_{a}\circ d\psi_{b}(\partial_t v)+ \partial_t a\cdot Z(u)+\partial_{t} b \cdot R(u)
\end{cases}$$
where $d\rho_{a}$ and $d\psi_{b}$ denote the spatial derivatives. Then the desired equation $\partial_{s}u+J_{S}(u)\partial_{t}u=0$ is equivalent to:
\begin{equation}\label{eqn:holomorphic-lift}
  d\rho_{a}d\psi_{b}(\partial_s v+ J \partial_t v)+\left(\partial_s a-\partial_t b\right)\cdot Z(u)+\left(\partial_s b+\partial_t a\right)\cdot R(u)=0.
\end{equation}
Define the connection $\xi$ on the bundle $\pi:SY\to M$ to be the symplectic orthogonal to the fibres. The flows $\rho$ and $\psi$ preserve this connection and $J_{S}$ is preserved under the splitting by virtue of its compatibility with $\omega$. Decompose $\partial_{s}v$ and $\partial_{t}v$ according to $TSY=\xi\oplus \R Z\oplus \R R$. After some manipulation, and because we assume that $\bar{u}$ was holomorphic to begin with, \eqref{eqn:holomorphic-lift} reduces to a system of two equations:
$$\begin{cases}
        \partial_sa-\partial_t b= h_1,\\
        \partial_s b+\partial_t a=h_2.
\end{cases}$$
This is an inhomogenous Cauchy-Riemann problem:
$$\bar{\partial}f=h_1+ih_2.$$
Such a solution exists, one simply recalls that the Cauchy-Riemann operator on maps $f:(D,\partial D)\to (\C, i\R)$ is surjective.\footnote{It is a Fredholm operator of index $1$, and $\dim \ker \bar{\partial}=1$, thus its cokernel vanishes.} The boundary conditions easily prove that $\ker \partial f=i\R$, and hence any other solution is of the form $f+i c$ where $c\in \R$ is a constant, as desired.
\end{proof}

\subsection{Regularity of pearl trajectories}
\label{sec:lift-pearl-traj}

We prove Lemma \ref{lemma:BK-is-regular}. The proof is similar in spirit to the proof in \S\ref{sec:morse-smale-property}. Let us denote by $\mathscr{M}_{d}(z_{-},z_{+})$ the moduli space of pearl trajectories $\bar{w}$ from $z_{-}$ to $z_{+}$ appearing in moduli spaces with $\mathrm{vdim}(\bar{w})=d$ and $\mathfrak{n}=0$. We only consider $d\le 1$. Because we assume $(Q,J)$ is very regular in the sense of Remark \ref{remark:non-eq-regular}, these moduli spaces are manifolds of dimension $d$ (after quotient by an $\R^{k+1}$-action), except in the case $d=-1$ and $k=0$ when the moduli space consists only of the constant trajectories.

We consider pearl trajectories $w=(b,\pi,u_{1},\dots,u_{k})$ for $(P_{\eta},J_{S})$ for the pair $(SY,L)$ with $\mathrm{vdim}(w)=0$ and $\dim \mathscr{P}=j$ which project to pearl trajectories $\bar{w}=(b,\bar{u}_{1},\dots,\bar{u}_{k})$ in $\mathscr{M}_{d}(z_{-},z_{+})$. By the neck-stretching argument in \S\ref{sec:proof-of-neck-stretching}, such trajectories $w$ compute the quantum cohomology differential for sufficiently neck stretched data.

We first enumerate the curves which can appear. We treat the inverse images of $\mathscr{M}_{d}$ for $d=-1,0,1$. The inverse image over $\mathscr{M}_{-1}(z,z)$ can be analyzed directly using Morse theory on the circle. The only trajectories in this inverse image (with $\mathrm{vdim}(w)=0$) are those without pearls and are in components:
\begin{enumerate}[label=(\alph*)]
\item\label{reg-a} $y(z)$ to $x(z)$ with $\dim \mathscr{P}=2$,
\item\label{reg-b} $x(z)$ to $x(z)$ and $y(z)$ to $y(z)$ with $\dim \mathscr{P}=1$.
\end{enumerate}

Next the inverse image over $\mathscr{M}_{0}(z_{-},z_{+})$ consists of three types of trajectories:
\begin{enumerate}[label=(\alph*),resume]
\item\label{reg-c} $x(z_{-})$ to $x(z_{+})$ with $\dim \mathscr{P}=0$,
\item\label{reg-d} $y(z_{-})$ to $y(z_{+})$ with $\dim \mathscr{P}=0$,
\item\label{reg-e} $y(z_{-})$ to $x(z_{+})$ with $\dim \mathscr{P}=1$.
\end{enumerate}
Finally we consider the inverse image over $\mathscr{M}_{1}(z_{-},z_{+})$:
\begin{enumerate}[label=(\alph*),resume]
\item\label{reg-f} $y(z_{-})$ to $x(z_{+})$ with $\dim \mathscr{P}=0$.
\end{enumerate}
The counts of these trajectories are determined in \S\ref{sec:comp-diff}. In this section, we provide the arguments explaining why we can assume these are regular.

\subsubsection{Regularity in cases \ref{reg-a} and \ref{reg-b}}
\label{sec:regularity-cases-a-b}

These cases follow easily from, e.g., the usual Sard-Smale argument as in \cite[\S3.2]{mcduff-salamon-book-2012}. The fact that the vector field $V_{\eta}$ appearing in $P_{\eta}=Q+V_{\eta}$ is arbitrary when $\eta$ is far from poles implies that (provided $\pi$ is not constant) the universal operator obtained by varying $V_{\eta}$ will be surjective. Here we appeal to similar analysis to \S\ref{sec:morse-smale-property}.

\subsubsection{Regularity in cases \ref{reg-c}, \ref{reg-d}, \ref{reg-f}}
\label{sec:regul-cases-c-d}

This occurs solely in the $\dim \mathscr{P}=0$ component, and so is a ``non-equivariant'' phenomenon, and is therefore covered by the work of \cite{biran-khanevsky-CMH-2013}. The argument we give below in \S\ref{sec:regul-cases-e} can be adapted to show the present case, and so we move directly onto the next case for the sake of saving space.

\subsubsection{Regularity in case \ref{reg-e}}
\label{sec:regul-cases-e}

The same argument given in \S\ref{sec:regularity-cases-a-b} works to establish the $k=0$ case, so we assume $k>0$ in this subsection.

The logic is similar to \S\ref{sec:morse-smale-property}. The choice of solution $w$ projecting to a solution $(\bar{w}, \pi)\in \mathscr{M}_{0}(z_{-},z_{+})\times \mathscr{P}$ forms a sort of circle bundle. Since $k>0$, there is always a first pearl $\bar{u}_{1}$ and we trivialize this bundle by evaluation at $\bar{u}_{1}(-\infty)$ and pulling back $\bar{L}\to L$. The fiber over $(\bar{w},\pi)$ is identified with the choice of lift $u_{1}(-\infty)$. There are two canonical sections $w_{x}$ and $w_{y}$ for any $(\bar{w},\pi)$, the lift which ends at $x$ and the lift which starts at $y$, and comparing these sections give a smooth $\R/\Z$ function $f_{z_{-},z_{+}}$.

\textbf{Claim.} {\itshape For generic choice of extension $P_{\eta}$, the subset $w_{x}=w_{y}$ is transversally cut codimension 1 submanifold $\mathscr{R}(z_{-},z_{+})\subset \mathscr{M}_{0}(z_{-},z_{+})\times \mathscr{P}$.}

The argument is straightforward application of the Sard-Smale theorem; this uses that the flow lines $\pi\in \mathscr{P}$ pass through the region where we are allowed to vary $P_{\eta}$. Then $\mathscr{M}_{0}(y(z_{-}),x(z_{+}))$ is identified with $\mathscr{R}(z_{-},z_{+})$ via the projection map.

The rest of the argument is similar to \S\ref{sec:morse-smale-property}. Let us recall that the linearized operator of \S\ref{sec:linearized-operators} is defined on:
\begin{equation*}
  D_{w}:(\bigoplus_{\nu=1}^{k}W^{1,p,\delta}(u_{\nu}^{*}TSY,u_{\nu}^{*}TL))\oplus \mathscr{I}_{w}\to \bigoplus_{\nu=1}^{k}L^{p,\delta}(u^{*}_{\nu}TSY),
\end{equation*}
where $\mathscr{I}_{w}$ is a finite dimensional vector space of coincident variations. To make a long story short, the domain and codomain of $D_{w}$ split into a part tangent to $\R Z\oplus \R R \simeq \C$ and a part tangent to the distribution $\xi$. The operator is upper triangular with respect to this splitting. Following the same strategy as in \S\ref{sec:morse-smale-property}, regularity boils down to:

\textbf{Claim.} {\itshape The kernel of the restriction of $D_{w}$ to the vertical part is zero.}

This is sufficient because then $\dim \ker D_{w}$ is bounded from above by: $$\dim \mathscr{R}(z_{-},z_{+})=\mathrm{index}(D_{w}),$$ using the same argument of differentiating the projection to $\mathscr{R}(z_{-},z_{+})$. To see that the kernel of the vertical restriction is zero, we appeal to the unique lifting \S\ref{sec:proof-of-unique-lifting} of disks. By analysis of the asymptotic convergence to $x(z_{+})$ in future time, it follows that any variation which is purely vertical must be eventually zero (i.e., the vertical tangent vector at the last pearl $u_{k}(+\infty)$ must be zero). Then, by unique lifting, it follows that the tangent vector at $u_{k}(-\infty)$ is also zero, and by \emph{coincident variations} (see \S\ref{sec:linearized-operators}) the variation at $u_{k-1}(+\infty)$ is also zero, etc. In this fashion, one concludes that the entire variation is zero, as desired. This completes the proof. \hfill$\square$

\subsection{Computation of the differential and Theorem \ref{theorem:differential}}
\label{sec:comp-diff}

Abbreviate:
\begin{itemize}
\item $x=x(z_{-})$, $y=y(z_{-})$, $z=z_{-}$,
\item $x'=x(z_{+})$, $y'=y(z_{+})$, $z'=z_{+}$.
\end{itemize}
It is sufficient to compute the outputs $d_{eq}(x)$ and $d_{eq}(y)$, which give four coefficients:
\begin{equation*}
  \left[
    \begin{matrix}
      A&C\\
      B&D
    \end{matrix}\right]=\left[
    \begin{matrix}
      \ip{d_{eq}(x),x'}&\ip{d_{eq}(y),x'}\\
      \ip{d_{eq}(x),y'}&\ip{d_{eq}(y),y'}
    \end{matrix}\right];
\end{equation*}
each coefficient is valued in $\mathbf{k}[q^{-1},q]]\otimes \mathbf{k}[e]$. The rest of the proof is dedicated to computing these four coefficients.

First of all, write $A=A_{0}+eA_{1}+e^{2}A_{2}+\cdots$. The count $A_{0}$ equals the non-equivariant quantum differential, which by the arguments of \S\ref{sec:lift-pearl-traj} (see also \cite[Section 7]{biran-khanevsky-CMH-2013}), coincides with the count $\ip{d(z),z'}\in \mathbf{k}[q^{-1},q]]$ where $d$ is the ordinary quantum differential on $\mathrm{CM}(Q)\otimes \mathbf{k}[q^{-1},q]]$.

Consider now the term $A_{j}$ with $j\ge 1$, and let $w$ be a solution of \S\ref{sec:defin-diff} joining $x$ to $x'$. By the dimension formula \eqref{eq:virtual-dimension-differential} one gets (since $\mathrm{vdim}(w)=0$):
$$\mu(w)+\mathrm{ind}_{P}(x')-\mathrm{ind}_{P}(x)+j-1=0.$$
Recalling that $\mathrm{ind}_{P}(x')=\mathrm{ind}_{Q}(z')$ and $\mathrm{ind}_{P}(x)=\mathrm{ind}_{Q}(z)$ it follows that:
$$\mu(w)+\mathrm{ind}_{Q}(z')-\mathrm{ind}_Q(z)+j-1=0.$$
Since the projection of $w$ (appealing to Lemma \ref{lemma:neck-stretching}) is a regular solution to \S\ref{sec:defin-diff}, one sees that $\mathrm{vdim}(u)=-j$, and so the only solutions are the constant solutions with $\mathrm{vdim}(u)=-1$ when $j=1$. In this case, $x'=x$.

We have therefore shown $A=d+eA_{1}$ and $A_{1}=\ip{d_{1}^{+}(x),x}+\ip{d_{1}^{-}(x),x}.$ To complete the proof we show the two terms in $A_{1}$ cancel each other. Indeed, these terms correspond to counting Morse trajectories over the gradient lines $\pi(s)$ joining $v_{0,+}$ to $v_{1,\pm}$. Each count must in fact be $1$ because it is a coefficient in a continuation map $\mathrm{CM}(P_{+}|S^{1})\to \mathrm{CM}(P_{\pm}|S^{1})$ where $P_{+}=P$ and $P_{-}=a^{*}P$; this coefficient must be nonzero because the restriction of the pseudogradient to the circle has only two critical points $x,y$. Thus $A_{1}=1+1=0$.

Next we consider the case $B=B_{0}+eB_{1}+e^{2}B_{2}+\dots$. The count $B_{0}$ is zero by comparison with Morse theory of a circle; indeed, by index considerations the moduli space is empty unless the projection is constant and hence $z=z'$. In this case the count is still zero, as it corresponds to a coefficient in the Morse differential of our pseudogradient restricted to a circle, (the differential vanishes because there are only two generators $x,y$ and it computes the homology of a circle).

Let us show $B_{j}= 0$ for $j>0$. Suppose there is a solution $w$ from $x$ to $y'$ solving \S\ref{sec:defin-diff} with $\mathrm{vdim}(w)=0$. By using the index formula we have:
$$\mu(w)+\mathrm{ind}_{P}(y)-\mathrm{ind}_{P}(x)+j=1$$
which comparing to indices on the base would imply:
$$\mu(w)+\mathrm{ind}_{Q}(z')-\mathrm{ind}_{Q}(z)=-j,$$
which does not happen by regularity of $(Q,J)$. Thus $B_{j}=0$ for $j>0$.

Next we consider the case of $C=C_{0}+eC_{1}+e^{2}C_{2}+\dots$. The count of $C_{0}$ defines an endomorphism $F$; moreover this endomorphism $F$ is easily seen to coincide with the definition of the connecting homomorphism from \cite[Theorem 4.1.1]{biran-khanevsky-CMH-2013} (in their notation, it would correspond to picking a lift $y$ of $z$ under their ``$p$'', applying the differential and then pulling back $x'$ to $z'$ via their map ``$i$'' to the base).

The vanishing of the $C_{1}$ term is somewhat subtle. We write it as:
\begin{equation*}
  C_{1}=C_{1,+}+C_{1,-}=\ip{d_{1,+}(y),x'}+\ip{d_{1,-}(y),x'}.
\end{equation*}
By index considerations, any solution $w$ with $\mathrm{vdim}(w)=0$ contributing to this term must live above (i.e., project to) an pearly trajectory $\bar{w}$ from $z$ to $z'$ which contributes to the quantum differential (so $\bar{w}$ lies in the one-dimensional moduli space of its translations). We will deform the equation $w$ solves and show there is a compact cobordism between the solutions defining $C_{1,+}$ and $C_{1,-}$, and therefore $C_{1}=0$. The arguments are somewhat similar to the arguments used in consideration of the BV-operator in Floer cohomology; see \cite{abouzaid-EMS-2015}.

Let $[\pi_{\pm}]$ be the two flow lines (modulo translation) joining $v_{0,+}$ to $v_{1,\pm}$, and consider:
\begin{equation*}
  P_{\pm,s}=P_{\pi_{\pm}(s)}.
\end{equation*}
Note that $P_{\pm,s}=P_{+}$ for $s$ sufficiently negative, while for $s$ sufficiently positive it holds that $P_{+,s}=P_{+}$ and $P_{-,s}=P_{-}$. By the formulas $P_{+}=P_{\epsilon,\theta}$, and $P_{-}=P_{\epsilon,\theta+0.5}$, we can therefore define an interpolation $P_{\lambda}=P_{\epsilon,\theta+0.5\lambda}$ and define a family over $\lambda\in [0,1]$:
\begin{equation*}
  P_{\lambda,s}=\left\{
    \begin{aligned}
      &P_{+}\text{ for $s$ sufficiently negative},\\
      &P_{\lambda}\text{ for $s$ sufficiently positive},\\
      &P_{+,s}\text{ for $\lambda=0$ },\\
      &P_{-,s}\text{ for $\lambda=1$};
    \end{aligned}
  \right.
\end{equation*}
we suppose $P_{\lambda,s}$ is related to $Q$. Note that the changing positive asymptotic $P_{\lambda}$ remains in the conjugacy class of $P$ under the circle action; in particular, if one treats the asymptotic ends separately, the analysis of the equation is equivalent to one where the asymptotics are fixed.

Let us denote a representative underlying solution as $\bar{w}=(b;\bar{u}_{1},\dots,\bar{u}_{k})$ and write $b=(b_{1},\dots,b_{k})\in \mathscr{G}_{k}$. This leads to a parametric moduli space $\mathscr{M}(\bar{w})$ of solutions over the parameter $\lambda\in [0,1]$ and $\sigma\in \R$, given by data $(\lambda,\sigma;u_{1},\dots,u_{k})$ where:
\begin{itemize}
\item each $u_{i}$ is in the symplectization end, is $J_{S}$-holomorphic, and projects to $\bar{u}_{i}$,
\item the flow line of $P_{\lambda,s}$ on the interval $[\sigma+b_{i},\sigma+b_{i+1}]$ takes $u_{i}(+\infty)$ to $u_{i+1}(-\infty)$,
\item the flow lines on the end intervals similarly take $u_{1}(-\infty)$ and $u_{k}(+\infty)$ to $y$ and $x'$.
\end{itemize}
Once we show that the solutions in the one-dimensional component of this moduli space form a compact cobordism between the counts for $\lambda=0$ and $\lambda=1$, we will be done, as the counts for $\lambda=0,1$ are $C_{1,\pm}$ by construction.

Let us observe that each $u_{i}$ appearing in a solution $w\in \mathscr{M}(\bar{w})$ must have a bounded first derivative; indeed, $u_{i}$ must remain the circle family of $J_{S}$-holomorphic lifts of $\bar{u}_{i}$ by Lemma \ref{lemma:unique-lifting}. Therefore the only possible failure of compactness is if $\sigma\to\pm\infty$. However, if $\sigma\to\pm\infty$, one would conclude a solution to the quantum differential equation with fixed pseudogradient $P_{\pm}$ with negative virtual dimension, which cannot happen for a generic choice. Thus $\mathscr{M}_{1}$ forms a compact cobordism and $C_{1}=0$.

In the case $C_{j}$ with $j\geq 3$, an index analysis rules out the existence of solutions $w$. Furthermore, this analysis also implies that the $e^2$ term consists only of equivariant Morse trajectories contained in the $S^{1}$ fibre over $z=z'$. Let us package the count of trajectories which remain in the $S^{1}$ fibre over $z$, with zero pearls, as defining a differential $d_{loc}$ on the complex:
$$\mathrm{CM}_{eq}(S^{1})\simeq \mathbf{k}[e]x\oplus \mathbf{k}[e]y.$$
Then, the count of trajectories $C_{2}$ represents the $\ip{d_{loc}(y),x}.$

The differential $d_{loc}$ differential is just the Morse version of the equivariant cohomology differential, and standard arguments show that its homology $\mathrm{HM}_{eq}(S^{1})$ is given by:
\begin{lemma}\label{lemma:equivariant-circle}
  There is an isomorphism:
  $$\mathrm{HM}_{eq}(S^1)\cong \mathbf{k}[e]\slash (e^2)$$
  of $\mathbf{k}[e]$-modules.
\end{lemma}
It follows by a dimension count that $\ip{d_{loc}(y),x}=e^{2}$ since all higher $e$ terms vanish (and the $e$ terms $\ip{d_{loc}(x),x}=\ip{d_{loc}(y),y}=0$ vanish).

Finally we claim $D=d$. This is analogous to the case of $A$ with no real modifications needed.\hfill$\square$

\begin{proof}[Proof of Lemma \ref{lemma:equivariant-circle}]
  As in \S\ref{sec:cont-isom}, one can show that $H^{\ast}_{eq}(S^1)$ is independent of the perturbation data. Consider the following vector field on $S^{1}$:
  $$V_{t}=-\cos(2\pi t)\sin(2\pi t)\partial_{t}.$$
  Note that it is $\Z/2\Z$-equivariant with respect to the involution. It has two index $0$ critical points $x_0=0, x_1=0.5$ and two index $1$ critical points $y_0=0.25$ and $y_1=0.75$. The usual Morse differential can be computed and is equal to:
  $$\begin{cases}
    d_{0}(x_i)=y_0+y_1,\\
    d_{0}(y_i)=0.
  \end{cases}.$$
  Now we consider equivariant contributions. However, the vector field is equivariant and all relevant moduli spaces are transversely cut-out: hence the Borel data $P_{\eta}\equiv V$ for all $\eta\in S^{\infty}$ is a regular perturbation datum.

  We compute the $e$ term: this corresponds to constant solutions, but also their twist by the involution. Hence:
  $$\begin{cases}
    d_{1}(x_j)=d_{1}^{+}(x_j)+d_{1}^{-}(x_j)=x_0+x_1,\\
    d_{1}(y_j)= y_0+y_1.
  \end{cases}$$
  Higher equivariant terms vanish for index reasons. Now construct the $k[e]$-linear map defined by:
  \begin{eqnarray*}
    \phi: k[e]\slash (e^2)& \mapsto & H^{\ast}_{eq}(S^1)\\
    1&\mapsto & [x_0+x_1]
  \end{eqnarray*}
  Note that $x_0+x_1$ is indeed a cycle, so the assignment makes sense. To see that it is well defined, we compute
  $$\phi(e^2)=[e^2(x_0+x_1)]=[e\cdot e(x_0+x_1)]=[e\cdot d_{eq}(x_0)]=[0].$$
  Any class in $H^{\ast}_{eq}(S^1)$ is cohomologous to either $[x_0+x_1]$ or $[e(x_0+x_1)]$. Indeed, in degree $0$ any cycle is of the form $\alpha x_0+\beta x_1$ and by inspection this will be closed if and only if $\alpha=\beta=1$. Moreover, any degree $1$ cycle can be written $\alpha y_0+\beta y_1+e(\gamma x_0+\chi x_1)$, and for this expression to be a cycle one needs either $\alpha=\beta=\gamma=\chi=1$ or one of the pairs $(\alpha,\beta)$, $(\gamma,\chi)$ is identically zero. But then note that $y_0+y_1=e(x_0+x_1)+\delta(x_0)$, which shows that the above class is either zero or cohomologous to $[e(x_0+x_1)]$. This shows that the map is surjective.

  To see that $\phi$ is injective note that by degree reasons $x_0+x_1$ cannot be a coboundary. This concludes the proof.
\end{proof}

\subsection{Proof of Theorem \ref{theorem:main-computation}}
\label{sec:proof-theorem-main}

We show the isomorphism:
$$\mathrm{QH}^{\ast}_{eq}(L)\cong \left(\mathrm{QH}^{\ast}(\bar{L})\otimes \mathbf{k}[e]\right)\slash(e^2+F),$$
of modules over $\mathbf{k}[q^{-1},q]]\otimes \mathbf{k}[e]$, where $F$ is as in \S\ref{sec:comp-diff}. Keeping the splitting of the generators on the equivariant chain complex in mind (as in the proof of Theorem \ref{theorem:differential}), we define a map:
\begin{equation*}
  i:z\in \mathrm{CM}(Q)\mapsto x(z)\in \mathrm{CM}(P),
\end{equation*}
which we extend to be a module homomorphism over $\mathbf{k}[q^{-1},q]]\otimes \mathbf{k}[e]$ via extension of coefficients. By Theorem \ref{theorem:differential}, $i$ is a chain map provided one uses the non-equivariant quantum differential on the domain (with coefficients extended by tensoring $\mathbf{k}[e]$), and the equivariant quantum differential on the codomain. There is then an induced map on homology:
\begin{equation}
  i:\mathrm{QH}^{\ast}(\bar{L})\otimes \mathbf{k}[e]\to \mathrm{QH}^{\ast}_{eq}(L).
\end{equation}
Let us compute the kernel. If a cycle $\zeta$ is sent to an exact cycle, then there are $\alpha,\beta\in\mathrm{CM}(Q)\otimes\mathbf{k}[q^{-1},q]]\otimes\mathbf{k}[e]$, considered as a tuple $(\alpha,\beta)$, such that:
$$d_{eq}(\alpha,\beta)= \begin{bmatrix}
        d\alpha +(e^2+F)\beta\\
        d\beta
\end{bmatrix}=\left[\begin{matrix}
  \zeta\\
  0
\end{matrix}\right].$$
The second coordinate shows that $\beta$ is a cycle for the non-equivariant differential, and the first implies that $\zeta$ and $(e^2+F)\beta$ are homologous. Hence the homology class $[\zeta]$ lies in the image of $(e^2+F)$; the same argument in reverse shows that the kernel of $i$ is equal to the image of $(e^{2}+F)$. Thus $i$ descends to a well-defined injective map:
$$i:\mathrm{QH}^{\ast}(\bar{L})\otimes \textbf{k}[e]\slash(e^2+F)\to \mathrm{QH}^{\ast}_{eq}(L).$$

To show surjectivity of $i$, we inductively show that any cycle $(\alpha,\beta)$ in the equivariant chain complex is homologous to a cycle of the form $(\zeta,0)$. The induction is on the degree $\deg(\alpha)$ (as a polynomial in the variable $e$) of $\alpha$. First use Theorem \ref{theorem:differential}: to show $d_{eq}(\alpha,\beta)=0$ implies:
\begin{equation*}
  d\alpha+(e^2+F)\beta=0\implies \deg(\beta)+2\le \deg(\alpha).
\end{equation*}
This immediately gives the base cases $\det(\alpha)=0,1$ cases of the induction. In general, for $m=\deg(\alpha)$ with $m\ge 2$, write $\alpha=\alpha_{0}+\dots+\alpha_{m}e^{m}$, and then replace:
\begin{equation*}
  (\alpha',\beta')=(\alpha,\beta)+d_{eq}(0,-\alpha_{m}e^{m-2}),
\end{equation*}
which has $\deg(\alpha')<\deg(\alpha)$, and hence the induction proceeds.\hfill$\square$

\section{The case of free actions and the proof of Theorem \ref{theorem:free-action}}
\label{sec:proof-theorem-free-action}

Let $W$ and $L$ be as in Theorem \ref{theorem:free-action}; here we recall that the $\Z/2\Z$ action on $W$ is free, with smooth quotients $W'=W\slash\Z/2\Z$ and $L'=L\slash \Z/2\Z$.

Consider a regular pair of data $(P',J')$ for defining $\mathrm{QH}^{*}(P',J')$ in the sense of Remark \ref{remark:non-eq-regular}. This lifts to $\Z/2\Z$ equivariant data $(P,J)$ on $W$. It is almost immediate to show that this lifted data is regular in the sense of \S\ref{sec:nodal-curv}: indeed, the linearized problem at a solution for $(P,J)$ is identified with the linearized problem at a solution for $(P',J')$ by simply projecting $W\to W'$.

Above each zero $z$ of $P'$ live two zeros of the same index, $x(z)$ and $y(z)$; unlike the proof of Theorem \ref{theorem:differential}, there is no canonical way to distinguish these (and so we pick an arbitrary labeling). Define the operator $a$ which swaps $x(z)$ with $y(z)$. Then:
\begin{equation*}
  d_{eq}=\bar{d}+e(1+a),
\end{equation*}
where $\bar{d}$ is defined over $\mathrm{CM}_{*}(P)\otimes \mathbf{k}[q^{-1},q]]$ and is the specialization $e=0$ of the equivariant differential. The matrix entries of $\bar{d}$ cannot be determined explicitly because of the arbitrary choice of $x(z),y(z)$; any trajectory $w$ from a zero $z$ to a zero $z'$ has exactly two lifts, one starting at $x(z)$ and one starting at $y(z)$, and they land in $\set{x(z'),y(z')}$, however we cannot control which one. What we do know is that if one lift ends in $x(z')$, the other ends in $y(z')$ and vice-versa. In particular, one has the formula:
$$d_{eq}(x(z)+y(z))=\bar{d}(x(z)+y(z))=x(dz)+y(dz);$$
here $d$ is the non-equivariant differential for $P',L'$. In particular, this shows that $\phi(z):=x(z)+y(z)$ is a chain map. Theorem \ref{theorem:free-action} is then implied by:
\begin{lemma}
  The map $\phi$ is a quasi-isomorphism.
\end{lemma}
\begin{proof}
  We prove it is injective and surjective on homology. To prove injectivity, suppose $\phi(z)=d_{eq}(\beta)$ for some $\beta=\beta_0+\dots+e^{n}\beta_n$. We prove $\zeta$ is exact by induction on $n$. If $n=0$, then: $$d_{eq}(\beta)=\bar{d}(\beta_0)+e(1+a)\beta_0=\phi(z).$$ Since $\phi(z)$ has no $e$ term, this implies $(1+a)\beta=0$, and hence $\beta$ lies in the image of $\phi$, and it follows that $\zeta$ is exact. Now assume the result holds for $n$ and let $\beta=\beta_0+\dots+ \beta_{n+1}e^{n+1}$ be such that $d_{eq}\beta=\phi(z)$. It follows, by comparing $e$ terms, that:
  $$\begin{cases}
    (1+a)\beta_{n}+\bar{d}\beta_{n+1}=0,\\
    (1+a)\beta_{n+1}=0.
  \end{cases}$$
  In particular, $\beta_{n+1}$ is symmetric, and hence can be written as: $$\beta_{n+1}=(1+a)\gamma_{n}.$$ Let $\tilde{\beta}=\beta_{0}+\dots+ e^{n}(\beta_{n}+\bar{d}\gamma_{n})=\beta+d_{eq}(e^{n}\gamma_{n})$. Then $d_{eq}(\tilde{\beta})=d_{eq}\beta=\phi(\zeta)$, and by our induction hypothesis it follows that $\zeta$ is exact, as desired.
  
Next we establish surjectivity. Let $\gamma=\gamma_0+\dots+ e^{n}\gamma_n$ be a cycle in the equivariant chain complex. We prove that $\gamma$ is cohomologous to a cycle of the form $\phi(\zeta)$ by induction on $n$. If $n=0$, then $d_{eq}\gamma_0= \bar{d}\gamma_0+e(1+a)\gamma_0=0$ implies that $\gamma_0$ is symmetric and thus is in the image of $\phi$. Assume the induction hypothesis for $n$ and fix a cycle $\gamma=\gamma_0+\dots+\gamma_{n+1}e^{n+1}$. The fact that $\gamma$ is closed implies:
$$\begin{cases}
  \bar{d}\gamma_{n+1}+(1+a)\gamma_n=0,\\
  (1+a)\gamma_{n+1}=0.
\end{cases}$$
Now choose $\delta_{n}$ so that $(1+a)\delta_{n}=\gamma_{n+1}$. Then $\gamma-d_{eq}(e^{n}\delta_{n})$ is a cycle, homologous to $\gamma$, of degree $n$. The result follows.
\end{proof}

\appendix

\section{Transversality for pearl trajectories}
\label{sec:proof-theorem-regular-b-data}

The goal is to prove Theorem \ref{theorem:regular-b-data}, which concerns the generic regularity of Borel data. We will prove generic regularity for all solutions in \S\ref{sec:defin-diff} and solutions of the nodal equation \S\ref{sec:nodal-curv}, by induction on the virtual dimension (the induction terminates at $\mathrm{vdim}(w)=1$). Recall the virtual dimension is given by:
\begin{equation*}
  \mathrm{vdim}(w)=\mu(w)+\mathrm{ind}(x_{+})-\mathrm{ind}(x_{-})+\dim\mathscr{P}-1-\mathfrak{n},
\end{equation*}
where $\mathfrak{n}$ is the number of nodes. If $\mathrm{vdim}(w)$ is sufficiently negative, there are no solutions at all, and so all solutions are regular.

Let us call a trajectory $w=(b,\pi,u_{1},\dots,u_{k})$ \emph{simple} provided the following holomorphic curves are simple:\footnote{in the sense their injective points are dense in the domain, as in \cite{mcduff-salamon-book-2012}}
\begin{itemize}
\item the union of the holomorphic disks $u_{i}$ over those indices $i$ such that $\pi(b_{i})$ is near the starting pole $v_{0,+}$,
\item the union of the holomorphic disks $u_{i}$ over those indices $i$ such that $\pi(b_{i})$ is near the ending pole $v_{j,\pm}$,
\item the holomorphic curves $u_{i}$ such that $\pi(b_{i})$ is far from both the starting and ending pole.
\end{itemize}
Here we recall $\pi:\R\to S^{\infty}$ is the auxiliary flow line. The meaning of a point $\eta=\pi(b_{i})$ being near a pole $v_{j,\pm}$ is that $\pm \eta_{j}\ge 4/5$, as per the definition of Borel data in \S\ref{sec:borel-data}. In the worst case, $\pi$ is a constant flow line from $v_{0}^{+}$ to $v_{0}^{+}$; in this case simplicity implies the union of all the curves is required to be simple.

By adapting standard property that a suitable universal operator is a submersion at simple holomorphic curves (see \cite[Proposition 3.4.2]{mcduff-salamon-book-2012}), one shows that, for generic $J$, all simple pearl trajectories are regular\footnote{Here we use the fact that $J_{\eta}$ is allowed to vary in an $\eta$-dependent way when $\eta$ is far enough from the poles.} (in sense of Definition \ref{definition:regular}).

Morally, the strategy is now to show this automatically implies that all trajectories with $\mathrm{vdim}(w)\le 1$ are simple. However, in the actual argument we will require further perturbation of $J$.

Suppose $w$ is a non-simple trajectory: the idea, following \cite[\S3.2]{biran-cornea-arXiv-2007}, is to extract from $w$ an underlying simple trajectory. The key tool used in this part of the argument is the work of Lazzarini \cite{lazzarini-1,lazzarini-2}, which analyzes the structure of general $J$-holomorphic curves with boundary on a Lagrangian submanifold; see Figure \ref{fig:shattered-pearl}.

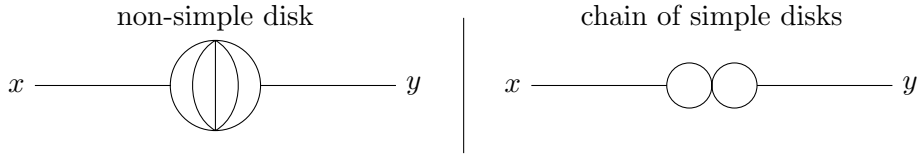
\begin{figure}[h]
  \centering
  \begin{tikzpicture}[scale=.6]
    \draw (0,0) circle (1) (0,-1)to[out=30,in=-30](0,1) (0,-1)to[out=150,in=-150](0,1) (0,-1)--(0,1) (1,0)--+(3,0)node[right]{$y$} (-1,0)--+(-3,0)node[left]{$x$} (0,1)node[above]{non-simple disk};
    \draw (5.5,1.5)--+(0,-3);
    \begin{scope}[shift={(11,0)}]
      \draw (-0.5,0) circle (.5) (0.5,0) circle (.5) (1,0)--+(3,0)node[right]{$y$} (-1,0)--+(-3,0)node[left]{$x$} (0,1)node[above]{chain of simple disks};      
    \end{scope}
  \end{tikzpicture}
  \caption{A non-simple pearl ``shatters'' into open holomorphic pieces separated by a certain graph $G$ by the results of \cite{lazzarini-1,lazzarini-2}; moreover, the interior of each holomorphic piece is a multiple cover of the interior of an underlying simple holomorphic disk. Therefore one can extract an underlying chain of simple disks as explained in \cite[\S3.2]{biran-cornea-arXiv-2007}.}
  \label{fig:shattered-pearl}
\end{figure}

To make a long story short:
\begin{lemma}
  For generic $J$ (the complex structure at the poles), a non-simple trajectory $w$ from $x$ to $y$ with $\mathrm{vdim}(w)=d$ implies the existence of a simple trajectory $w'$ from $x$ to $y$ with: $$\mathrm{vdim}(w')\le d-2.$$ In particular, if all simple trajectories $w'$ are regular, than all trajectories with $\mathrm{vdim}(w)\le 1$ are regular.
\end{lemma}
\begin{proof}
  We refer the reader to the proof of \cite[\S3]{biran-cornea-arXiv-2007} for the details. Let us briefly indicate the argument in one case: if the union of disks $u_{i}$ with $\pi(b_{i})$ close to the starting pole $v_{0,+}$ is not a simple holomorphic curve, then the disks ``overlap'' each other (e.g., one disk can overlap itself and become non-simple, see Figure \ref{fig:shattered-pearl}). This union of disks is $u_{1},\dots,u_{i}$ for some maximal $i$, and one considers the partial trajectory from $u_{1}(-\infty)$ to $u_{i}(+\infty)$. Using the simple pieces in Lazzarini's decomposition \cite{lazzarini-1,lazzarini-2}, one can extract a ``short-cut'' trajectory from $u_{1}(-\infty)$ to $u_{i}(+\infty)$. In the construction of this short-cut, the Maslov index must drop by at least $2$. By attaching this to the rest of the old trajectory $w$, one obtains the desired trajectory $w'$. 
\end{proof}

This completes the proof of Theorem \ref{theorem:regular-b-data}.\hfill$\square$

\subsection{Very regular continuation/deformation data exists}
\label{sec:very-regular-continuation-deformation-data-exists}

The argument is analogous to the arguments used in the proof of Theorem \ref{theorem:regular-b-data}.

\section{Elementary argument proving Theorem \ref{theorem:application} when $n=1$}
\label{sec:elem-argum-prov}

In this section we provide an elementary argument proving that, if $L\subset \C$ is the unit circle, then $L$ cannot be displaced from itself, or from $\R\subset \C$, by an equivariant Hamiltonian isotopy. In fact, we can prove:
\begin{proposition}\label{proposition:low-dimensional-1}
  Suppose that $\varphi_{\eta}$ is a smooth family of area preserving diffeomorphisms parametrized by $\eta\in S^{2}$ so that:
  \begin{equation}\label{eq:family-eq}
    \varphi_{-\eta}=a\varphi_{\eta}a.
  \end{equation}
  Then $\varphi_{\eta}(L)\cap L\ne\emptyset$ holds for at least two points $\eta\in S^{2}$.
\end{proposition}
\begin{proof}[Proof of Proposition \ref{proposition:low-dimensional-1}]
  Of course, it suffices to prove the statement for at least one point $\eta\in S^{2}$, since then $\eta$ and $-\eta$ will be the desired two points.

  Consider $L$ as the image of the standard embedding $\gamma:\R/\Z\to \C$. The family $\varphi_{\eta}$ yields a map $S^{2}\times \R/\Z\to \C$ given by:
  \begin{equation*}
    G(\eta,t)=\varphi_{\eta}(\gamma(t)).
  \end{equation*}
  Notice that this map is equivariant in the following sense:
  \begin{equation*}
    G(-\eta,t+1/2)=a\varphi_{\eta}(\gamma(t))=aG(\eta,t),
  \end{equation*}
  where we use that the standard embedding satisfies $\gamma(t+1/2)=a\gamma(t)$.

  In search of a contradiction, suppose that $G(\eta,t)\not\in L$ for all $\eta,t$.

  Consider the standard angular closed one-form:
  \begin{equation*}
    \alpha=\frac{x\d y-y\d x}{x^{2}+y^{2}}.
  \end{equation*}
  It is well-known the the integrals of $\alpha$ over closed loops are valued in $2\pi \Z$. We claim that the integral of $\alpha$ over each loop $t\mapsto \gamma_{\eta}(t)=G(\eta,t)$ is zero. To show this, we claim the loop $\gamma_{\eta}$ bounds a surface $\Omega_{\eta}$ disjoint from $0$, and so we can apply Stokes' Theorem to the closed one form $\alpha$. Clearly there exists a smooth surface $\Omega_{\eta}$ whose boundary is the image of $\gamma_{\eta}$ --- we can just take $\Omega_{\eta}=\varphi_{\eta}(D^{2})$ --- and this surface $\Omega_{\eta}$ has area $\pi$. If $\Omega_{\eta}$ contains $0$, then $\Omega_{\eta}$ must be entirely contained in the interior of $D^{2}$ as it is not allowed to intersect $L=\bd D^{2}$. But this would imply the area of $\Omega_{\eta}$ is strictly smaller than $\pi$, which is a contradiction. This establishes the claim.

  Thus $G^{*}\alpha$ is an exact one-form, since its integral vanishes over every loop in the domain $S^{2}\times \R/\Z$. Thus $G^{*}\alpha=dF$ for some real-valued function $F$. We next claim that:
  \begin{equation*}
    F(-\eta,t+1/2)=F(\eta,t)+\pi+2\pi k.
  \end{equation*}
  for some $k\in \Z$. Indeed, $F(-\eta,t+1/2)-F(\eta,t)$ is the integral of $\alpha$ over an arc joining a point $z\in \C$ to its antipodal point $-z$ (this arc lies in the image of $G$ and is therefore disjoint from $0$). It is clear that the integral of $\alpha$ over an arc joining a pair of antipodal points is valued in $\pi+2\pi k$ (one simply interprets $\alpha$ as the angular form). The integer $k$ is independent of the point $(\eta,t)$. However, this conclusion is impossible, because then:
  \begin{equation*}
    F(\eta,t)=F(-\eta,t+1/2)+\pi+2\pi k=F(\eta,t)+2\pi+4\pi k\implies 0=2\pi +4\pi k,
  \end{equation*}
  which cannot hold for any integer $k$. This completes the argument.
\end{proof}

\bibliographystyle{./amsalpha-doi}
\bibliography{citations}
\end{document}